\documentclass[reqno,11pt]{amsart}

\usepackage[all]{xy}
\CompileMatrices
\usepackage{amssymb}
\usepackage{mathrsfs}
\usepackage{cite}
\usepackage{amsfonts}
\usepackage[active]{srcltx}

\renewcommand{\to}{\longrightarrow}
\newcommand{\CCC}{k}

\newcommand{\rad}{\mathop{\mathrm{rad}}}

\newcommand{\J}{\mathrel{\mathscr J}} % J - relation
\newcommand{\D}{{\mathscr D}} % D - relation
\newcommand{\R}{\mathrel{\mathscr R}} % R - relation
\newcommand{\eL}{\mathrel{\mathscr L}} % L - relation

\newcommand{\pv}[1]{\mathbf {#1}}
\newcommand{\inv}{^{-1}}
\newcommand{\p}{\varphi}

\newcommand{\ov}[1]{\ensuremath{\overline {#1}}}
\newcommand{\til}[1]{\ensuremath{\widetilde {#1}}}

\newcommand{\wh}{\widehat}
\newcommand{\Irrm}{\mathop{\mathsf{Irr}}\nolimits}
\newcommand{\AIrr}{\mathop{\mathsf{AIrr}}\nolimits}
\newcommand{\Irr}{\mathop{\mathrm{Irr}}\nolimits}

\newcommand{\module}[1]{#1\text{-}\mathrm{mod}}
\newcommand{\Hom}{\mathop{\mathrm{Hom}}\nolimits}
\newcommand{\Ext}{\mathop{\mathrm{Ext}}\nolimits}
\newcommand{\Der}{\mathop{\mathrm{Der}}\nolimits}
\newcommand{\IDer}{\mathop{\mathrm{IDer}}\nolimits}

\newcommand{\nup}[1]{{#1}^{\not\hskip 1pt\uparrow}}
\newcommand{\Null}{\mathop{\mathrm{Null}}\nolimits}
\newcommand{\St}{\mathop{\mathrm{St}}\nolimits}

\newtheorem{Thm}{Theorem}[section]
\newtheorem{Prop}[Thm]{Proposition}

\newtheorem{Lemma}[Thm]{Lemma}
{\theoremstyle{definition}
\newtheorem{Def}[Thm]{Definition}}
{\theoremstyle{remark}
\newtheorem{Rmk}[Thm]{Remark}}
\newtheorem{Cor}[Thm]{Corollary}
{\theoremstyle{remark}
\newtheorem{Example}[Thm]{Example}}

{\theoremstyle{remark}
}
{\theoremstyle{remark}
\newtheorem{Case}{Case}}
{\theoremstyle{remark}
}

\numberwithin{equation}{section}

\title{Quivers of monoids with basic algebras}

\author{Stuart Margolis\ and Benjamin Steinberg}
\address{Department of Mathematics\\
Bar Ilan University\\ 52900 Ramat Gan \\ Israel \and Center for Algorithmic and Interactive Scientific Software\\ CCNY \\ CUNY \\ New York\\ New York \and School of Mathematics and Statistics \\ Carleton University \\
1125 Colonel By Drive\\
Ottawa, Ontario  K1S 5B6 \\
Canada}
\thanks{The first author wishes to warmly thank the
Center for Algorithmic and Interactive Scientific Software, CCNY, CUNY for inviting him to be a Visiting Professor during the preparation of this paper. The second author was supported in part by NSERC}
\email{margolis@math.biu.ac.il\and bsteinbg@math.carleton.ca}
\date{\today}

\keywords{quivers, representation theory, monoids, Hochschild-Mitchell cohomology, EI-categories}
\subjclass[2000]{20M25,16G10,05E99}
\begin{document}
\begin{abstract}
We compute the quiver of any finite monoid that has a basic algebra over an algebraically closed field of characteristic zero.  More generally, we reduce the computation of the quiver over a splitting field of a class of monoids that we term rectangular monoids (in the semigroup theory literature the class is known as $\pv{DO}$) to representation theoretic computations for group algebras of maximal subgroups.  Hence in good characteristic for the maximal subgroups, this gives an essentially complete computation.  Since groups are examples of rectangular monoids, we cannot hope to do better than this.

For the subclass of $\R$-trivial monoids, we also provide a semigroup theoretic description of the projective indecomposable modules and compute the Cartan matrix.
\end{abstract}
\maketitle

\section{Introduction}

Whereas the representation theory of finite groups has played a central role in that theory and its applications for more than a century, the same cannot be said for the representation theory of finite semigroups. While the basic parts of representation theory of finite semigroups were developed in the 1950s by Clifford, Munn and Ponizovksy,~\cite[Chapter 5]{CP}, there were not ready made applications of the theory, both internally in semigroup theory and in applications of that theory to other parts of mathematics and science. Thus, in a paper in 1971,~\cite{McAlistersurvey}, the only application mentioned of the theory was Rhodes's use of it to compute the Krohn-Rhodes complexity of a completely regular semigroup~\cite{KRannals,Rhodeschar}.

Over the past few years, this situation has changed. This comes from a number of sources. One is the theory of monoids of Lie type developed by Putcha, Renner and others~\cite{Putcha1,Putcha2,putchasemisimple,LieType}. These monoids can be thought of as finite analogues of linear algebraic monoids~\cite{Putcha,Renner} and their representation theory gives information on their groups of units, which are groups of Lie type~\cite{wonderfulcompact,Putcharep1,Lietypeandgrouprep}.  On a similar note, applications of semigroup representation theory to Schur-Weyl duality can be found in~\cite{Solomonrook,SchurWeyl}.

A second source is the theory of quasi-hereditary algebras~\cite{quasihered}. The algebras of finite (von Neumann) regular monoids provide natural and diverse examples of quasi-hereditary algebras. This was first proved by Putcha~\cite{Putcharep3} and further developed by the two authors of this paper using homological methods~\cite{rrbg}.  However, Nico essentially had noted that semigroup algebras of regular semigroups are quasi-hereditary before the concept was even invented.  Namely, he found a filtration by hereditary ideals (the same used by Putcha) and used it to bound the global dimension of the algebra~\cite{Nico1,Nico2}.

The third source came from applications in related areas, especially in automata and formal language theory, probability theory, algebraic combinatorics and the theory of Coxeter groups and related structures.

In~\cite{AMSV}, Almeida, Volkov and the authors of this paper computed the Rhodes radical of a finite monoid $M$, which is the congruence obtained by restricting the Jacobson radical from the monoid algebra $kM$ of $M$ over a field $k$ to $M$. They used this to solve a number of problems in automata and formal language theory. In particular, they determined the finite monoids $M$ whose algebras $kM$ are split basic over a field $k$ and proved that this collection is a variety of finite monoids in the sense of Eilenberg and Sch\"utzenberger~\cite{Eilenberg,Pinbook,qtheor,Almeida:book}. Recall that a $k$-algebra is split basic if all its irreducible representations are one-dimensional, or equivalently, the algebra has a faithful representation by triangular matrices over $k$. These monoids, and the generalization we dub \emph{rectangular monoids}, and their algebras are the central object of study in this paper.  Further applications of semigroup representation theory to automata theory and transformations semigroups can be found in~\cite{mortality,transformations}.

The faces of a central hyperplane arrangement have the structure of a left regular band, that is, a semigroup satisfying the two identities $x^{2}=x$ and $xy=xyx$, something that was taken advantage of by Bidigare \emph{et al.}~\cite{BHR} and Brown~\cite{Brown1,Brown2} to compute spectra of random walks on hyperplane arrangements, and by Brown and Diaconis~\cite{DiaconisBrown1} to compute stationary distributions and rates of convergence, as well as to prove diagonalizability for these walks.  See~\cite{Brown1,Brown2,DiaconisAthan,bjorner1,bjorner2} for further examples of applications of left regular band algebras to probability.

In particular, the Coxeter complex of any Coxeter group has the structure of a left regular band. Whereas the product in this complex goes back to Tits~\cite{Tits}, it was not until recently that its structure as a semigroup was exploited.  In fact, Tits thought of this product geometrically as a projection.  Formally, he associates to each face $A$ a unary operation $\mathrm{proj}_A$ which sends a face $B$ to its projection onto $A$. In his appendix~\cite{Titsappendix} to Solomon's paper~\cite{SolomonDescent}, he proves that
\[\mathrm{proj}_A(\mathrm{proj}_B(C)) = \mathrm{proj}_{\mathrm{proj}_A(B)}(C).\]
If one defines $AB = \mathrm{proj}_A(B)$, then the last formula is just the
associative law $A(BC)=(AB)C$. Perhaps this notational problem delayed the serious
study of the semigroup theoretic aspects of this product.
Thus, Ken Brown's first edition of his book {\em Buildings}~\cite{Brown1} does not mention the semigroup structure at all, whereas it plays a prominent role in the second edition~\cite{Brown2}. Bidigare also discovered that if one takes the reflection arrangement associated to a finite reflection group $W$, then the $W$-invariants of the algebra of the associated left regular band is precisely Solomon's descent algebra; see~\cite{Brown2} for details.   This led Aguiar \emph{et al}~\cite{Aguiar} to develop an approach to the representation theory of finite Coxeter groups via left regular bands.  Saliola computed the quiver, first of hyperplane face semigroups~\cite{Saliolahyperplane}, and then of a left regular band algebra, and the projective indecomposable modules~\cite{Saliola}; for the former he used homological methods, whereas in the latter case he computed primitive idempotents. Saliola also computed the quiver associated to the descent algebra of a Coxeter group in types $A$ and $B$ via its incarnation as the algebra of invariants of the Coxeter group acting on the algebra of the hyperplane face semigroup~\cite{SaliolaDescent}.  The result in type $A$ was also obtained by Schocker~\cite{schocker}, again using hyperplane face semigroups.  The quiver of a descent algebra in general was computed in~\cite{descentquiver} via other means.

Hsiao~\cite{Hsiao} developed an extension of these results to groups that are wreath products of a finite group with a symmetric group. He showed that the wreath product analogue of Solomon's descent algebra, the Mantaci-Reutenauer descent algebra~\cite{MantReut}, was an algebra of invariants of a left regular band of groups, a monoid whose idempotents form a left regular band and such that every element belongs to a subgroup of $M$. The algebras of such monoids were studied by the authors in~\cite{rrbg}. Hsiao's semigroup, in characteristic zero, does not have a basic algebra unless the group in the wreath product is abelian.  It does belong to the class, we shall introduce, of rectangular monoids. The case of left regular band of groups plays a major role in the current paper.

A second very important monoid associated to a Coxeter group is the monoid associated to its $0$-Hecke algebra. Norton first described the representation theory of the $0$-Hecke algebra of a Coxeter group $W$ in 1979~\cite{Norton}, but did not exploit its structure as the monoid algebra of a monoid $M(W)$; see also the work of Carter~\cite{Carter0Hecke}. The monoid $M(W)$ has been rediscovered many times over the years. The easiest way to define it is as the monoid with generating set the Coxeter generators of $W$ and relations those of $W$ in braid form, and replacing the involution relation $s^{2}=1$ by the idempotent relation $s^{2}=s$ for each Coxeter generator $s$. The monoid $M(W)$ has a number of amazing properties. Its size is exactly that of $W$ and it admits the strong Bruhat order as a partial order compatible with multiplication. In fact, it is isomorphic to the monoid of principal order ideals of the Bruhat order under set multiplication. It also is isomorphic to the monoid structure of Schubert and Bruhat cells of a reductive group~\cite{Richardson}.

From the point of view of this paper, we are interested that the monoid algebra of $M(W)$ is the $0$-Hecke algebra $\mathcal{H}_0(W)$, as can readily be seen from the presentation for $M(W)$ that we mentioned above. The monoid $M(W)$ belongs to the very important class of $\J$-trivial monoids. These are monoids $M$ such that different elements generate different principal ideals of $M$, or in the language of semigroup theory, Green's relation~\cite{Green} $\J$ is the identity relation~\cite{CP,qtheor,Eilenberg,Pinbook}. A $\J$-trivial monoid has a split basic algebra over any field. A number of authors~\cite{Fayers,Denton,Jtrivialpaper,BergeronSaliola,SchockerRtrivial,Kiselmanstuff,Duchampetal} have exploited the monoid structure to elucidate the representation theory of the $0$-Hecke algebra. Our results for rectangular monoids has the case of $\J$-trivial monoids as a very special case.

The purpose of the present paper is to study the representation theory of a common generalization of left regular bands and $\J$-trivial monoids, and indeed a generalization of the class of monoids that have basic algebras. Although this class has been studied by semigroup theorists under the name $\pv{DO}$~\cite{Almeida:book} (which stands for monoids whose regular $\D$-classes are orthodox semigroups, using the language of semigroup theory), we call them here {\em rectangular monoids}. The name reflects the fact that these are the monoids such that each conjugacy class of idempotents (or in semigroup parlance, every $\D$-class of idempotents) is a subsemigroup, which by elementary semigroup theory implies that the conjugacy classes of idempotents of rectangular monoids belong to the class of rectangular bands.

It is a fundamental result, known as the Munn-Ponizovsky Theorem, that the irreducible representations of any finite monoid are parameterized uniquely by a pair consisting of a conjugacy class $\mathcal{C}$ of idempotents and an irreducible representation of a maximal subgroup $G$ at any idempotent in $\mathcal{C}$. See~\cite{myirreps} for a modern proof of this fact. In general, computing the irreducible representation of $M$ from the pair $(\mathcal{C},G)$ is quite complex, requiring an induction process from an irreducible representation of $G$ followed by the computation of and the quotient by the radical of the induced module.

The key representation theoretic property of rectangular monoids is that every irreducible representation of a rectangular monoid $M$ is computed via a retraction onto a maximal subgroup $G$ followed by an irreducible representation of $G$, that is, as an inflation of an irreducible representation of $G$. It follows that every matrix representation of $M$ is then a block triangular monoid of matrices with the blocks being representations of maximal subgroups of $M$ (together with zero). Furthermore, the semisimple image of the algebra of $M$ (in good characteristic) is the direct sum of the algebras of the maximal subgroups of $M$, one summand for each conjugacy class of idempotents of $M$.  In good characteristic this effectively reduces the semisimple part of the representation theory of $M$ to the character theory of its maximal subgroups.

The class of rectangular monoids contains, in addition to the left regular bands and $\J$-trivial monoids mentioned above, many other important classes of monoids (in addition to all finite groups) that have arisen in the literature. These include all $\R$-trivial and $\eL$-trivial monoids~\cite{Eilenberg,qtheor,BergeronSaliola,Pinbook,Almeida:book}, all finite bands and the right and left regular bands of groups studied by the authors in~\cite{rrbg} (and by Hsiao~\cite{Hsiao}).  Schocker studied $\R$-trivial monoids under the name ``weakly ordered semigroups''~\cite{SchockerRtrivial}, as observed by the second author and recorded in~\cite{BergeronSaliola}.  The representation theory of EI-categories~\cite{WebbEI,Luckbook,Webb} can also be viewed as a special case of the representation theory of rectangular monoids. For applications of rectangular monoids to circuit and communication complexity, see~\cite{DenisDO1,DenisDO}.  In~\cite{DiaconisSteinberg}, random walks on regular rectangular monoids are studied, and in particular random walks on Hsiao's monoid are shown to model certain card-shuffling schemes.

Our main goal is to compute the (Gabriel) quiver of the algebra of a rectangular monoid. For rectangular monoids with trivial subgroups (like bands and $\J,\R,\eL$-trivial monoids), the computations essentially reduce to counting the number of equivalence classes for certain equivalence relations on certain subsets of these monoids. For general rectangular monoids $M$, the computation reduces, in good characteristic, to decomposing certain representations of products of maximal subgroups of $M$. Of course, in characteristic zero, this is classically done via character theory. We remark that for $\J$-trivial monoids results were recently obtained in~\cite{Jtrivialpaper} that match ours, but we use different tools that apply to all rectangular monoids.

The main tool we use is the well known identification of the number of arrows of the quiver between two vertices $U,V$ of irreducible $M$-modules over a field $k$ with the dimension of the first Hochschild cohomology space $H^{1}(M,\Hom_k(U,V))$~\cite{CartanEilenberg}. In turn, this space is identified with the space of outer derivations of $M$ in $\Hom_k(U,V)$~\cite{CartanEilenberg}. The structure of rectangular monoids allows for a complete description of this latter space modulo the representation theory of its maximal subgroups, which we take as known.

The paper is organized as follows. First we provide some background on finite dimensional algebras and review Hochschild cohomology and its connection to computing quivers via derivations. Next we note that the cases of bands and more generally, the class of (von Neumann) regular rectangular monoids (known as orthogroups in the semigroup literature~\cite{Almeida:book}) reduce almost immediately to the case of left and right regular bands studied in~\cite{Saliola} in the case of bands, and for regular rectangular monoids with non-trivial subgroups to that of right and left regular bands of groups studied in~\cite{rrbg}. More precisely, we show that the image of a regular rectangular monoid $M$ under the natural map $f\colon M \rightarrow kM/\rad^2(kM)$ is a subdirect product of left and right regular bands  (known as a regular band in the literature~\cite{Almeida:book}) if $M$ is a band, and is a subdirect product of left and right regular bands of groups for general regular rectangular monoids. Since for any finite dimensional algebra $A$, the quiver of $A$ and $A/\rad^{2}(A)$ are the same, the result follows rapidly.

The following section describes our results on quivers, without proof, when restricted to special cases of rectangular monoids that have been considered in the literature: bands, $\J$-trivial monoids, regular rectangular monoids and $\R$-trivial monoids.  Examples are provided.  In~\cite{BergeronSaliola}, a complete set of orthogonal primitive idempotents was constructed for the algebra of an $\R$-trivial monoid.  However, a semigroup theoretic description of the projective indecomposable modules and a computation of the quiver and Cartan matrix were lacking.  For $\J$-trivial monoids, these were computed in~\cite{Jtrivialpaper}.  Here we compute the quiver of an $\R$-trivial monoid, construct the projective indecomposable modules as partial transformation modules (in the sense of~\cite{transformations}) and compute the Cartan matrix, generalizing the case of left regular bands and $\J$-trivial monoids~\cite{Saliola,Jtrivialpaper}.  Our proof is independent of the construction of the primitive idempotents from~\cite{BergeronSaliola}, and in fact primitive idempotents play no role in our proof.

We then proceed to study derivations on rectangular monoids $M$ into arbitrary $M$-bimodules $A$ that are lifted from a pair of maximal subgroups of $M$ and this leads us to our main results. This proof is useful even for the case of bands and regular rectangular monoids mentioned above, as it elucidates the results obtained in~\cite{Saliola, rrbg}. We then establish the results announced in the previous section.

We are very aware that readers of this paper come from different communities of researchers. Representation theorists and algebraic combinatorialists will not necessarily be familiar with basic results of semigroup theory, whereas semigroup theorists may not necessarily be familiar with the representation theory and cohomology theory used here. We have done our best to make the paper understandable to both audiences. In particular, we have avoided as much as possible the use of Green's relations and the Rees Theorem, fundamental results of semigroup theory. Readers with a background in semigroup theory will certainly see where the use of such could simplify some arguments and are invited to do so. However, we felt strongly that proofs that could be understood by a general audience would be useful.

In the literature, many basic results of semigroup theory have been reproved in special cases in many papers applying semigroup theoretic methods. In particular, Brown's notion of the support lattice of a band~\cite{Brown1,Brown2} is nothing more than the maximal semilattice image of a band. The notion of maximal semilattice image goes back at least to Clifford's paper of 1941~\cite{Clifford}, a result that in particular applies to the case of bands. The maximal semilattice image of an arbitrary semigroup has been studied and applied in literally hundreds of semigroup related papers.
See for example,~\cite[Chapter 4]{CP}. Despite this usage, a number of recent papers have rediscovered this notion in special cases~\cite{Brown1,Brown2,BergeronSaliola,Saliola}. As the maximal semilattice image, and more generally the maximal semilattice of groups image, of a rectangular monoid play a crucial part in this paper, we give a completely self-contained construction of the maximal semilattice image of an arbitrary finite monoid.

\section{Finite dimensional algebras and Hochschild-Mitchell cohomology}
In this paper, $k$ will always denote a field.  Algebras over $k$ will be tacitly assumed finite dimensional and unital, except for a brief appearance of semigroup algebras.  We will work with finitely generated left modules, unless otherwise stated.  The category of finitely generated left $A$-modules for a $k$-algebra $A$ will be denoted $\module{A}$.

\subsection{Basic algebras and quivers}
A finite dimensional $k$-algebra is \emph{split basic} if $A/\rad(A)\cong k^n$ for some $n\geq 1$.  In other words, every irreducible representation of $A$ is one-dimensional, or equivalently, $A$ is isomorphic to an algebra of upper triangular matrices over $k$.  When $k$ is algebraically closed, then every finite dimensional algebra is Morita equivalent to a unique (up to isomorphism) (split) basic algebra~\cite{assem}. More generally, a finite dimensional $k$-algebra $A$ is said to \emph{split} over $k$ if $A/\rad(A)$ is isomorphic to a product of matrix algebras over $k$, that is, each irreducible representation of $A$ over $k$ is absolutely irreducible.  If $A$ splits over $k$, then $A$ is Morita equivalent to a unique split basic algebra.

The (Gabriel) \emph{quiver} $Q(A)$ of a finite dimensional $k$-algebra $A$, which is split over $k$, is the directed graph whose vertex set is the set of isomorphism classes of simple left $A$-modules.  The number of arrows from $S_i$ to $S_j$ is $\dim \Ext^1_k(S_i,S_j)$.

If $Q$ is a quiver (equals directed graph), then the path semigroup $P(Q)$ of $Q$ consists of all paths in $Q$ together with a multiplicative zero.  The product of paths is their concatenation, when defined, and otherwise is zero.  The path algebra $kQ$ is the contracted semigroup algebra of $P(Q)$. Recall that the contracted semigroup algebra of a semigroup with zero $z$, is the quotient of its usual semigroup algebra by the ideal generated by $z$~\cite{CP}.  The arrow ideal $J_Q$ of $kQ$ is the ideal generated by all edges of the quiver.  An ideal $I$ of $kQ$ is \emph{admissible} if, for some $n\geq 2$, $J_Q^n\subseteq I\subseteq J_Q^2$.  The algebra $kQ/I$ is then a finite dimensional split basic $k$-algebra with quiver $Q$.  On the other hand, the basic algebra associated to a finite dimensional algebra $A$ split over $k$ is $kQ(A)/I$ for some admissible ideal $I$~\cite{assem,benson}.  This is one motivation for computing the quiver of an algebra.

The quiver of a split basic algebra $B$ also can be viewed as an encoding of the $2$-dimensional representation theory of $B$ because a degree $2$ representation of $B$ is precisely an extension of a simple module by a simple module.

Throughout this paper, $M$ will always denote a finite monoid.  The opposite monoid is denoted $M^{op}$.  The monoid algebra of $M$ over $k$ is denoted $kM$.  We say that $k$ is a \emph{splitting field} for $M$, or that $M$ \emph{splits} over $k$, if $kM$ splits over $k$.  Throughout this paper, we shall always be working in the context of a fixed field $k$ and so by an \emph{$M$-bimodule} we mean a $kM$-bimodule.  Similarly, we speak of left and right $M$-modules.    We frequently view $M$-bimodules as $M\times M^{op}$-modules.  Also, we shall identify left $M^{op}$-modules with right $M$-modules.  For a group $G$, one has that $G$ is isomorphic to $G^{op}$ via inversion.  In particular, every right $G$-module $A$ can be viewed as a left $G$-module by putting $ga=ag\inv$.  Nonetheless, we will continue to use the notation $G^{op}$ when we are emphasizing that the action of $G$ is on the right.  If we want to view a $G^{op}$-module $A$ as a $G$-module via the inversion, then we will switch notation and explicitly refer to $A$ as a $G$-module.

If $G$ is a group and $U$ is a $G$-module, then the \emph{contragredient} module is $U^*=\Hom_k(U,k)$ with action given by $gf(u)=f(g\inv u)$ for $g\in G$ and $u\in U$.

As usual, if $\p\colon M\to M_n(k)$ is a linear representation, then the \emph{character} of $\p$ is the mapping $M\to k$ defined by $m\mapsto \mathrm{tr}(\p(m))$ where $\mathrm{tr}$ denotes the trace of a matrix.

\subsection{Hochschild-Mitchell cohomology}
We now summarize some results about Hochschild-Mitchell cohomology that will be relevant to computing quivers of monoids.

The idea of using Hochschild cohomology was motivated by trying to compute $\Ext^1_{kM}(S_1,S_2)$ where $S_1,S_2$ are simple $kM$-modules in order to compute the quiver of $kM$.  This amounts to describing up to equivalence the $M$-modules $V$ forming an exact sequence \[0\longrightarrow S_2\longrightarrow V\longrightarrow S_1\longrightarrow 0.\]  If $\pi_1,\pi_2$ are the matrix representations corresponding to $S_1$ and $S_2$, then the representation $\pi$ associated to $V$ has the block form
\[\begin{pmatrix} \pi_2 & d\\ 0 & \pi_1\end{pmatrix}\] where $d\colon M\to \Hom_k(S_1,S_2)$ satisfies $d(mn) = \pi_2(m)d(n)+d(m)\pi_1(n)$.  That is, $d$ is a derivation of $M$ in $\Hom_k(S_1,S_2)$.

Equivalent extensions are obtained via conjugation by a matrix of the form \[\begin{pmatrix} 1 & a \\ 0 &1\end{pmatrix}\] with $a\in \Hom_k(S_1,S_2)$.  The corresponding derivations then differ by an inner derivation, i.e., one of the form $m\mapsto \pi_2(m)a-a\pi_1(n)$.  Thus we have \[\Ext^1_{kM}(S_1,S_2)\cong \mathrm{Der}(\Hom_k(S_1,S_2))/\mathrm{IDer}(\Hom_k(S_1,S_2))\] where the right hand vector space is the space of derivations modulo inner derivations.  This is in fact the Hochschild-Mitchell cohomology space $H^1(M,\Hom_k(S_1,S_2))$.  Let us recall the definitions.

Let $A$ be an $M$-bimodule and put $C^n(M,A)=\{f\colon M^n\to A\}$ where we are just considering arbitrary maps $f\colon M^n\to A$.  Then $C^n(M,A)$ is a $k$-vector space.  Define the coboundary map $d^n\colon C^n(M,A)\to C^{n+1}(M,A)$ by
\begin{align*}
d^n(f)(m_1,\ldots, m_{n+1}) &= m_1f(m_2,\ldots, m_{n+1})\\ &+\sum_{i=1}^n(-1)^if(m_1,\ldots, m_im_{i+1},\ldots,m_{n+1})\\ &+(-1)^{n+1}f(m_1,\ldots,m_n)m_{n+1}
\end{align*}
One can verify that $(C^{\bullet}(M,A),d^n)$ is a chain complex. The \emph{Hochschild-Mitchell cohomology} of $M$ with coefficients in $A$ is defined by \[H^{\bullet}(M,A)=H^{\bullet}(C^\bullet(M,A));\]
see~\cite[Chapter IX]{CartanEilenberg} and~\cite{ringoids}.  The Hochschild-Mitchell cohomology is defined in general for additive categories, but for monoids the chain complex simplifies to the form given above.

For instance, $d^0(a)(m) = ma-am$ and \[d^1(f)(m_1,m_2)=m_1f(m_2)-f(m_1m_2)+f(m_1)m_2.\]  Thus \[H^0(M,A) =\{a\in A\mid ma=am, \forall m \in M\}\] is the space of \emph{$M$-invariants} of $A$.  A \emph{derivation} of $M$ in $A$ is a map $d\colon M\to A$ such that \[d(m_1m_2)= m_1d(m_2)+d(m_1)m_2.\]  A derivation is \emph{inner} if it is of the form $d_a$ for some $a\in A$, where $d_a(m) = ma-am$.  It is easy to see that $1$-cocyles are precisely derivations and $1$-coboundaries are inner derivations.  Denoting the space of derivations by $\Der(M,A)$ and that of inner derivations by $\IDer(M,A)$, we have
\[H^1(M,A)=\Der(M,A)/\IDer(M,A).\]  Often elements of $\Der(M,A)/\IDer(M,A)$ are called \emph{outer derivations}.

It is known that \[H^{\bullet}(M,A)\cong \Ext^{\bullet}_{k[M\times M^{op}]}(kM,A)\] cf.~\cite[Chapter~IX]{CartanEilenberg}.

Let $V,W$ be $M$-modules.  Then $\Hom_k(V,W)$ is naturally an $M$-bimodule where $(mf)(v)=mf(v)$ and $(fm)(v)= f(mv)$.  The following theorem is~\cite[Chapter~IX, Corollary~4.4]{CartanEilenberg}.

\begin{Thm}\label{hochschildvsext}
Let $V,W$ be $M$-modules.  Then \[\Ext^{\bullet}_{kM}(V,W)\cong H^{\bullet}(M,\Hom_k(V,W)).\]
\end{Thm}

We have the following important special case, a direct proof of which we sketched earlier.
\begin{Cor}\label{derivationcomputation}
If $U,V$ are left $kM$-modules, then \[\Ext^1_{kM}(U,V) \cong \Der(M,\Hom_k(U,V))/\IDer(M,\Hom_k(U,V)).\]
\end{Cor}

Our strategy to compute quivers will then be to understand the space of outer derivations.

The following observation will be useful.

\begin{Prop}\label{grouphochiszero}
Let $G$ be a finite group and suppose that the characteristic of the field $k$ does not divide the order of $G$.  Then $H^n(G,A)=0$ for $n\geq 1$ and all $G$-bimodules $A$.
\end{Prop}
\begin{proof}
As $k[G\times G^{op}]$ is a semisimple algebra, the $G\times G^{op}$-module $kG$ is projective.  Thus $H^n(G,A)=\Ext^n_{k[G\times G^{op}]}(kG,A)=0$ for $n\geq 1$.
\end{proof}

\section{Rectangular monoids}
Let $M$ be a finite monoid.  The set of idempotents of $M$ is denoted $E(M)$; more generally, if $X\subseteq M$, then put $E(X)=E(M)\cap X$.  If $M$ is not a group, then $E(M)$ has at least two elements. Idempotents are partially ordered by setting $e\leq f$ if $ef=e=fe$.  If $m\in M$, then there is a unique idempotent in the subsemigroup generated by $m$; it is denoted $m^{\omega}$.  One always has $m^{\omega}=m^{n!}$ where $n=|M|$, cf.~\cite{Almeida:book}.

Recall that if $M$ is a monoid, then a (left) $M$-set is a set equipped with a left action of $M$ by mappings.  Following the same convention we have used for modules, right $M$-sets will be identified with $M^{op}$-sets.

Next, we recall the definition of the Karboubi envelope (also called Cauchy completion or idempotent splitting) of a monoid (the notion exists more generally for categories~\cite{Borceux1,ElkinsZilber}).
If $M$ is a monoid, the \emph{Karoubi envelope} of $M$ is the small category $\mathscr K(M)$ with object set $E(M)$ and arrow set consisting of all triples $(e,m,f)$ with $m\in fMe$.  The arrow $(e,m,f)$ has domain $e$ and codomain $f$.  Composition is given by \[(f,m,g)(e,n,f)=(e,mn,g).\]  The identity at the object $e$ is the triple $(e,e,e)$.  The hom set $\mathscr K(M)(e,f)$ can fruitfully be identified with the set $fMe$ by remembering only the middle coordinate of a triple.  Under this identification,  composition is given by the multiplication of $M$ and the identity at $e$ is $e$, itself.  We will in general avoid the cumbersome triple notation and write $m\colon e\to f$ if it is unclear from the context that we are thinking of $m$ as an element of $\mathscr K(M)(e,f)$. The following proposition is well known, but for lack of a precise reference, we include a proof.

\begin{Prop}\label{splitmonosKaroubi}
A morphism $m\in \mathscr K(M)(e,f)$ is a split monomorphism if and only if $Me=Mm$, a split epimorphism if and only if $fM=mM$ and an isomorphism if and only if $Me=Mm$, $fM=mM$.
\end{Prop}
\begin{proof}
Suppose that $m$ is a split monomorphism.  Then there is morphism $n\in \mathscr K(M)(f,e)$ such that $nm=e$.  Thus $Me\subseteq Mm$.  But $Mm\subseteq Me$ by definition of $\mathscr K(M)(e,f)$.  The case of split epimorphisms is dual and isomorphisms are precisely those morphisms that are simultaneously split epimorphisms and split monomorphisms.
\end{proof}

\begin{Rmk}
In the parlance of Green's relations~\cite{Green,CP,qtheor} one has that $m\in \mathscr K(M)(e,f)$ is a split monomorphism (epimorphism) if and only if $m\eL e$ ($m\R f$) and hence is an isomorphism if and only if $e\eL m\R f$.  Thus isomorphism classes of idempotents correspond to what are called $\mathscr D$-classes in the semigroup theoretic literature.
\end{Rmk}

The importance of the Karoubi envelope is that it is equivalent to the category of left $M$-sets of the form $Me$ with $e\in E(M)$, which is in fact equivalent to the category of projective indecomposable $M$-sets; see~\cite{ElkinsZilber}, where the situation is studied in the more general context of presheaves on a category.

The following is a piece of classical finite semigroup theory, cf.~\cite{Arbib,qtheor}, or see~\cite{transformations} for a presentation intended for combinatorialists.

\begin{Prop}\label{Dequiv}
Let $M$ be a finite monoid and let $A=kM$ with $k$ a field. Let $e,f\in E(M)$.
Then the following are equivalent:
\begin{enumerate}
\item $Me\cong Mf$ as left $M$-sets;
\item $eM\cong fM$ as right $M$-sets;
\item there exists $a,b\in M$ such that $ab=e$ and $ba=f$;
\item there exists $x\in fMe$ and $x'\in eMf$ such that $x'x=e$, $xx'=f$;
\item the objects $e$ and $f$ of $\mathscr K(M)$ are isomorphic;
\item $MeM=MfM$;
\item $Ae\cong Af$ as left $A$-modules;
\item $eA\cong fA$ as right $A$-modules;
\item $e=ufu\inv$ for some unit $u\in A$;
\item the two-sided ring-theoretic ideals generated by $e$ and $f$ in $A$ coincide;
\item all irreducible characters of $M$ over $k$ agree on $e$ and $f$.
\end{enumerate}
\end{Prop}

We say that the idempotents $e,f$ are \emph{conjugate} if the equivalent conditions of Proposition~\ref{Dequiv} hold. We remark that finiteness of $M$ is need for several parts of the equivalence. In semigroup parlance, we are essentially stating the fact that Green's relations $\J$ and $\D$ coincide for finite semigroups~\cite{Green}.

If $e\in E(M)$, then $eMe$ is a monoid with identity $e$ and hence has a group of units $G_e$, called the \emph{maximal subgroup} of $M$ at $e$.  One can check that, just as in ring theory, $eMe$ is the endomorphism monoid of the left $M$-set $Me$ (and the right $M$-set $eM$) and $G_e$ is its automorphism group, cf.~\cite{transformations}.  Thus if $e,f$ are conjugate idempotents, then $eMe\cong fMf$, whence $G_e\cong G_f$.

\subsection{The support lattice}
An important role in the representation theory of monoids with a basic algebra is played by the support lattice of the monoid.  Let us define the concept.

An \emph{ideal} of a monoid $M$ is a subset $I$ such that $MI\cup IM\subseteq M$.  We allow in this paper the empty set to be considered an ideal.  Each finite monoid has a unique minimal non-empty ideal, which is referred to as the \emph{minimal ideal} of $M$.  The minimal ideal is principal and is generated by a conjugacy class of idempotents, cf.~\cite{Arbib} or~\cite[Appendix A]{qtheor}.  An ideal $P$ is \emph{prime} if $M\setminus P$ is a submonoid.  In particular, $P$ must be proper.  We remark that the empty ideal is prime.  An arbitrary union of prime ideals is evidently prime, and so the set $\mathrm{spec}(M)$ of prime ideals is a lattice with respect to inclusion.

We view lattices as monoids via their meet operation. Every finite monoid $M$ admits a least congruence such that the quotient is a lattice: one takes the least congruence $\equiv$ such that $x\equiv x^2$ and $xy\equiv yx$ for all $x,y\in M$.  The corresponding quotient is called the \emph{maximal lattice image} of $M$.   Let $\pv 2=\{0,1\}$ be the two-element lattice (which can be viewed as the multiplicative monoid of the two-element field).  If $M$ is a finite monoid, let $\wh M$ be the set of monoid homomorphisms $\theta\colon M\to \pv 2$.  Notice that $\wh M$ is a finite lattice with respect to the pointwise ordering.  If $L$ is a lattice, then $L^{\varrho}$ will denote the lattice $L$ with the reverse ordering.

The following is piece of classical semigroup theory going back at least 70 years~\cite{Clifford}.  We sketch a proof using semilattice duality~\cite{CL,johnstone,latticeduality}.

\begin{Prop}\label{latticeimage}
Let $M$ be a finite monoid. Then:
\begin{enumerate}
\item\label{latticeimage.1} $\wh M\cong \mathrm{spec}(M)^{\varrho}$;
\item\label{latticeimage.2} if $L$ is a finite lattice, then $\wh L\cong L^{\varrho}$ and $\wh {\wh L}\cong L$ via the evaluation map;
\item\label{latticeimage.3} for each $m\in M$, there is a largest prime ideal $P(m)$ such that $m\notin P(m)$;
\item\label{latticeimage.4} the mapping $\sigma\colon M\to \mathrm{spec}(M)$ given by $\sigma(m)=P(m)$ is a homomorphism;
\item\label{latticeimage.5} given a homomorphism $\p\colon M\to L$ with $L$ a lattice, there is a unique homomorphism $\psi\colon \mathrm{spec}(M)\to L$ such that the diagram
    \[\xymatrix{M\ar[rr]^{\sigma}\ar[rd]_{\p} & &\mathrm{spec}(M)\ar@{-->}[ld]^{\psi}\\ & L &}\]
    commutes.
\end{enumerate}
\end{Prop}
\begin{proof}
To prove \eqref{latticeimage.1}, first observe that if $P$ is a prime ideal, then the characteristic function $\chi_{M\setminus P}$ is in $\wh M$.  Conversely, if $\theta\colon M\to \pv 2$ is a homomorphism, then $P=\theta\inv (0)$ is a prime ideal and $\theta=\chi_{M\setminus P}$.  Trivially, if $P$ and $P'$ are prime ideals, then $P\subseteq P'$ if and only if $\chi_{M\setminus P'}\leq \chi_{M\setminus P}$ and so $\wh M\cong \mathrm{spec}(M)^{\varrho}$.

Property \eqref{latticeimage.2} is almost immediate from \eqref{latticeimage.1}.  If $P$ is a prime ideal of the lattice $L$, then by finiteness, $L\setminus P$ has a minimum element $e$ and $L\setminus P$ is the principal filter $e^{\uparrow}=\{f\in L\mid f\geq e\}$. Conversely, if $e\in L$, then $L\setminus e^{\uparrow}$ is a prime ideal.  It follows that $\wh L\cong L^{\varrho}$ because $e^{\uparrow}\subseteq f^{\uparrow}$ if and only if $f\leq e$.  One easily deduces that $\wh {\wh L}\cong L$ via the evaluation map.  For future reference, if we want to view $\wh {\wh L}$ as $\wh L^{\varrho}\cong \mathrm{spec}(L)$, then $e\in L$ corresponds to the largest prime ideal $P$ of $L$ such that $e\notin P$, as $e^{\uparrow}$ is the smallest filter containing $e$.

To prove \eqref{latticeimage.3}, observe that the empty set is a prime ideal not containing $m$.  Since $\mathrm{spec}(M)$ is closed under union, it follows that there is a largest prime ideal $P$ with $m\notin P$ (namely the union of all such prime ideals).

We now prove \eqref{latticeimage.4} and \eqref{latticeimage.5} together.  Let $L$ be the maximal lattice image of $M$.  Then by the universal property of $L$, we must have $\wh M=\wh L$.  Thus \[L\cong \wh {\wh L}\cong \wh {\wh M}\cong \wh{\mathrm{spec}(M)^{\varrho}}\cong \mathrm{spec}(M).\]  By the remark two paragraphs up, the image of $m$ under the evaluation map $M\to \wh{\wh M}$ corresponds under the isomorphism $\wh {\wh M}\cong \mathrm{spec}(M)$ to the largest prime ideal not containing $m$, that is, to $P(m)$.
\end{proof}

It will be convenient to identity $\mathrm{spec}(M)$, for a finite monoid $M$, with a certain set of principal ideals. If $X$ is a principal ideal, we put \[M_X=\{m\in M\mid X\subseteq MmM\}\] and $\nup X=M\setminus M_X$.   Note that $\nup X$ is an ideal.  Let us say that $X$ is \emph{primary} if $\nup X$ is a prime ideal, or equivalently, if $M_X$ is a submonoid.  Let $\Lambda(M)$ be the poset of primary ideals, ordered by inclusion.

\begin{Prop}\label{primeideallattice}
The mapping $\tau\colon \Lambda(M)\to \mathrm{spec}(M)$ given by $X\mapsto \nup X$ is an order isomorphism.  Thus $\Lambda(M)$ is a lattice.
\end{Prop}
\begin{proof}
Trivially, if $X,Y$ are primary, one has $X\subseteq Y$ if and only if $\nup X\subseteq \nup Y$.  Thus it suffices to show that every prime ideal $P$ is of the form $\nup X$.  Let $J$ be the minimal ideal of $M\setminus P$ and put $X=MJM$.  As $J$ is principal in $M\setminus P$, clearly $X$ is principal.  Also $m\in M\setminus P$ if and only if $J\subseteq MmM$, if and only if $m\in M_X$.  This completes the proof.
\end{proof}

Thus, we can take the maximal lattice image of a finite monoid $M$ to be given by $\sigma\colon M\to \Lambda(M)$ where $\sigma(m)$ is the largest primary ideal $P_m$ such that $P_m\subseteq MmM$.
The following proposition is well known, see the proof of~\cite[Lemma 4.6.38]{qtheor}.

\begin{Prop}\label{Jclassisasemigroup}
Let $X$ be a principal ideal.  Then $X$ is primary if and only if it is generated by a conjugacy class $D$ of idempotents such that $e,f\in D$ implies $MefM=MeM$.
\end{Prop}
\begin{proof}
Suppose first that $X$ is primary.  Let $J$ be the minimal ideal of $M_X$.  The proof of Proposition~\ref{primeideallattice} shows that $X=MJM$.  Now $D=E(J)$ is a conjugacy class of idempotents of $M_X$ because $J$ is its minimal ideal.  However, since $\nup X$ is an ideal, it in fact follows that $D$ is a conjugacy class of idempotents of $M$.  Since $J$ is a subsemigroup of $M_X$, it follows that if $e,f\in D$, then $MefM=MJM=MeM$.

For the converse, suppose $m,n\in M_X$ and let $e\in D$. Then $e = umv = xny$ with $u, v, x, y\in M$.  Let $f=veum$. Then $f^2=ve(umv)eum=veum=f$.  Clearly $MfM\subseteq MeM$.  But also $umfv=um(veum)v = e$.  Thus $MfM=MeM$.  Similarly, $f'=nyex$ is idempotent and $Mf'M=MeM$.  Thus $f,f'\in D$ and so by hypothesis, $MeM=Mff'M=MveumnyexM$.  Therefore, $mn\in M_X$ and so $X$ is primary.
\end{proof}

In semigroup parlance, this says that the lattice of primary ideals is the lattice of regular $\D$-classes that are subsemigroups, equipped with the $\J$-order.

Following Brown's terminology for bands~\cite{Brown1,Brown2}, we shall call $\Lambda(M)$ the \emph{support lattice} of $M$ and the mapping $\sigma\colon M\to \Lambda(M)$ will be called the \emph{support map}.
It should be noted that the papers~\cite{Brown1,Brown2,Saliola,BergeronSaliola} all use the reverse ordering on $\Lambda(M)$ (and so are essentially working with $\wh M$).

The class of finite monoids in which each idempotent-generated ideal is primary, or equivalently, by Proposition~\ref{Jclassisasemigroup}, each regular $\D$-class is a subsemigroup, plays a distinguished role in semigroup theory.  It goes by the acronym $\pv{DS}$ --- ``regular $\D$-classes are subsemigroups'' --- in the literature; semigroups in this class are also known as (semi)lattices of archimedean semigroups.  This class was first considered independently by Sch\"utz\-en\-ber\-ger~\cite{Schutznonambig} and Putcha~\cite{PutchaDS}.  For example, a linear algebraic monoid with zero whose underlying variety is irreducible belongs to $\pv{DS}$ if and only if its group of units is solvable~\cite{Putcha,Renner}.

An important subclass of $\pv{DS}$ is the class of rectangular monoids.  We say that $M$ is a \emph{rectangular monoid} if each conjugacy class of idempotents is a subsemigroup.  In the literature, this class is known as $\pv{DO}$. Here $\pv{DO}$ stands for ``regular $\D$-classes are orthodox semigroups". An orthodox semigroup is a von Neumann regular semigroup whose idempotents form a subsemigroup. The reason for our terminology is that the conjugacy classes are a type of semigroup known in the literature as a rectangular band, i.e., a direct product of a left zero semigroup and a right zero semigroup.

The class of rectangular monoids is closed under taking finite direct products, submonoids, homomorphic images and taking the opposite monoid~\cite{Almeida:book}.  Every group and every band is a rectangular monoid, as are $\J$-trivial monoids and $\R$-trivial monoids.  The results of~\cite{AMSV} show that if $M$ is a monoid with a basic algebra over an algebraically closed field of characteristic zero, then $M$ is rectangular. More precisely, a monoid $M$ has a basic algebra over an algebraically closed field of characteristic zero if and only if $M$ is rectangular and all its maximal subgroups are abelian.

Since every conjugacy class of idempotents in a rectangular monoid $M$ is a subsemigroup, Proposition~\ref{Jclassisasemigroup} indeed implies that the primary ideals of $M$ are precisely the idempotent-generated ideals.   In particular, the support map $\sigma\colon M\to \Lambda(M)$ satisfies $\sigma(m)=\sigma(m^{\omega})=Mm^{\omega}M$.

A finite monoid $M$ is called a \emph{Clifford monoid} (because of~\cite{Clifford}) if it satisfies $x^{\omega}x=x$ and $x^{\omega}y^{\omega}=y^{\omega}x^{\omega}$ for all $x,y\in M$.  These are precisely the inverse monoids with central idempotents~\cite{Lawson}.
They are also called \emph{semilattices of groups} because they can all be constructed in the following manner.  One takes a finite lattice $\Lambda$ and a presheaf of groups $\{G_e\mid e\in \Lambda\}$ on $\Lambda$~\cite{MM-Sheaves}.  If $e\leq f$, then there is a restriction homomorphism $\rho^f_e\colon G_f\to G_e$.  The restriction maps of course satisfy the obvious compatibility conditions.  One then makes $\coprod_{e\in \Lambda}G_e$ into a monoid by defining the product, for $g\in G_e$ and $h\in G_f$, by $\rho^e_{e\wedge f}(g)\rho^f_{e\wedge f}(h)$.  See~\cite{Clifford,CP} for details.  Clifford monoids are rectangular monoids.  More generally, any monoid whose idempotents are central (e.g., a commutative monoid) is rectangular.

The following proposition provides various characterizations of rectangular monoids that are well known to semigroup theorists, cf.~\cite[Exercise 8.1.3]{Almeida:book}.

\begin{Prop}\label{DOcharacterization}
Let $M$ be a finite monoid.  Then the following are equivalent:
\begin{enumerate}
\item $M$ is a rectangular monoid;
\item $MeM=MfM\iff efe=e,\ fef=f$ for all $e,f\in E(M)$;
\item $e,f\in E(M)$ are conjugate if and only if $efe=e,\ fef=f$;
\item for all $x,y\in M$, one has $(xy)^{\omega}(yx)^{\omega}(xy)^{\omega}=(xy)^{\omega}$;
\item if $m,n\in M$, $e,f\in E(M)$ and $e\in MmM\cap MnM\cap MfM$, then $emfne=emne$;
\item if $m,n\in M$, $e\in E(M)$ and $MmM=MnM=MeM$, then $men=mn$.
\end{enumerate}
\end{Prop}

It follows from the results of~\cite{AMSV} that a finite monoid is rectangular if and only if it has a faithful representation over the complex numbers (or any algebraically closed field of characteristic zero) by block upper triangular matrices of the form
\[\begin{pmatrix}M_1 & \ast& \cdots &\ast\\ 0 & M_2 & \ast &\vdots\\ \vdots &0 &\ddots&\ast\\ 0&\cdots& 0&M_n  \end{pmatrix}\]
with $M_i\setminus \{0\}$ a group for $i=1,\ldots,n$.  This includes in particular, any finite monoid of upper triangular matrices over such a field.

The following proposition contains an important well-known fact about maximal subgroups of rectangular monoids.

\begin{Prop}\label{definerestrict}
Let $M$ be a rectangular monoid and let $e\in E(M)$. Let $X=MeM$.  Then there is a retraction $\rho_e\colon M_X\to G_e$ defined by $\rho_e(m)=eme$.  Moreover, if $f\in E(M)$ with $MfM\subseteq MeM$, then $\rho_f\rho_e=\rho_f|_{M_X}$.
\end{Prop}
\begin{proof}
First we verify that $\rho_e\colon M_X\to eMe$ is a homomorphism. Indeed, Proposition~\ref{DOcharacterization} immediately yields that if $m,n\in M_X$, then $\rho_e(m)\rho_e(n)=emene=emne=\rho_e(mn)$.  Also $\rho_e(1) =e$.  So $\rho_e\colon M_X\to eMe$ is a homomorphism.

For $m\in M_X$, we compute $\rho_e(m)^{\omega} = \rho_e(m^{\omega}) = em^{\omega}e = e$ by Proposition~\ref{DOcharacterization}.  Thus $\rho_e(m)$ is invertible in $eMe$, i.e., belongs to $G_e$.  Moreover, if $g\in G_e$, then $g\in M_X$ and $\rho_e(g)=ege=g$.  So $\rho_e$ is a retraction.

Finally, if $Y=MfM\subseteq X$, then $M_X\subseteq M_Y$.  If $m\in M_X$, then $\rho_f\rho_e(m) = femef=fmf=\rho_f(m)$ where the penultimate equality uses Proposition~\ref{DOcharacterization}. This shows that $\rho_f\rho_e=\rho_e|_{M_X}$.
\end{proof}

Since the maximal subgroup at an idempotent $e$ depends only on the conjugacy class of $e$, we can thus associate to each element $X\in \Lambda(M)$ a group $G_X$ that is uniquely determined up to isomorphism.  More specifically, we fix, for each ideal $X\in \Lambda(M)$, an idempotent $e_X$ with $Me_XM=X$ and define $G_X=G_{e_X}$.  Up to isomorphism, it depends only on $X$.  We call $G_X$ the \emph{maximal subgroup of $X$}.  It is also convenient to set $\rho_X=\rho_{e_X}$ for $X\in \Lambda(M)$.  Sometimes, it is valuable to think of $\rho_X$ as a homomorphism $\rho_X\colon M\to G_X\cup \{0\}\subseteq kG_X$ by sending $\nup X$ to $0$.

To each rectangular monoid $M$ (together with a fixed set $\{e_X\}_{X\in \Lambda(M)}$ of idempotent generators of the primary ideals of $M$), we can associate the monoid \[\Lambda_*(M)=\{(X,g)\mid X\in \Lambda(M),g\in G_X\}\] where the product is given by \[(X,g)(Y,h) = (X\wedge Y,\rho_{X\wedge Y}(gh)).\] The identity is $(M,1)$. Associativity relies on the fact that if $Y\subseteq X$, then $\rho_Y\rho_X=\rho_Y|_{M_X}$. We remark that $\Lambda_*(M)$ is a Clifford monoid.
One can prove that $\Lambda_*(M)$ depends only on $M$, and not the choice of idempotents, up to isomorphism, although we shall not use this.  In fact, $\Lambda_*(M)$ is the maximal Clifford monoid image of $M$.

The support map induces a surjective homomorphism $\sigma_*\colon M\to \Lambda_*(M)$ by putting \[\sigma_*(m) = (\sigma(m),\rho_{\sigma(m)}(m)).\]  This map is surjective because if $g\in G_X$, then $(X,g) = \sigma_*(g)$.  The verification that $\sigma_*$ is a homomorphism again boils down to the property $\rho_Y\rho_X=\rho_Y|_{M_X}$.

A homomorphism $\p\colon M\to N$ of monoids is called an \emph{$\pv{LI}$-morphism} if, for each idempotent $e\in N$, the semigroup $S_e=\p\inv(e)$ satisfies $fS_ef=f$ for all $f\in E(S_e)$.  The importance of $\pv{LI}$-morphisms for us stems from~\cite[Theorem~3.5]{AMSV}, which states that $\pv{LI}$-morphisms are the semigroup analogues of algebra homomorphisms with nilpotent kernels.

\begin{Thm}\label{radicalthm}
Let $\p\colon M\to N$ be an $\pv{LI}$-morphism of monoids and let $k$ be a field.  Then the kernel of the induced morphism $\ov{\p}\colon kM\to kN$ is nilpotent.  If the characteristic of $k$ is zero, the converse holds.
\end{Thm}

See~\cite{AMSV} for an analogue of Theorem~\ref{radicalthm} for characteristic $p$.

\begin{Prop}\label{isLImap}
The map $\sigma_*\colon M\to \Lambda_*(M)$ is an $\pv{LI}$-morphism.
\end{Prop}
\begin{proof}
An idempotent of $\Lambda_*(M)$ is of the form $(X,e_X)$.  Suppose that $m,f\in \sigma_*\inv(X,e_X)$ with $f\in E(M)$.  Then $\sigma(m)=\sigma(e_X)=\sigma(f)$ and $e_Xme_X=e_X=e_Xfe_X$.  Thus $fmf=fe_Xme_Xf=fe_Xf=f$ where we have applied Proposition~\ref{DOcharacterization} several times.
\end{proof}

A monoid $M$ is (von Neumann) \emph{regular} if $m\in mMm$ for all $m\in M$.  A regular monoid which is rectangular is called an \emph{orthogroup} in the literature.  Orthogroups can be characterized as those finite monoids $M$ satisfying $x^{\omega}x=x$ for all $x\in M$ and whose idempotents form a submonoid.  In particular, bands and left/right regular bands of groups are orthogroups.  In~\cite{DiaconisSteinberg}, random walks on orthogroups are analyzed, including card-shuffling examples.

\section{The irreducible representations of a rectangular monoid}\label{complexalgebra}
Fix, for the remainder of this section, a rectangular monoid $M$ as well as a set $\{e_X\}_{X\in \Lambda(M)}$ of idempotent generators for the idempotent-generated principal ideals of $M$.  The results of this section are not new.  They are specializations of the results of~\cite{AMSV,mobius2}.    The special case of a right regular band of groups was considered previously by the authors~\cite{rrbg}. The reader interested in more details about the representation theory of finite semigroups in general is referred to~\cite{myirreps} for a modern approach or to~\cite[Chapter 5]{CP} and~\cite{RhodesZalc,McAlisterCharacter,Putcharep5,McAlistersurvey} for the classical approach; see also~\cite{Putcharep3,rrbg} for connections with the theory of quasi-hereditary algebras and~\cite{AMSV,mobius2,mortality,transformations,DiaconisSteinberg,Putcharep1,Malandro1,Malandro2} for applications.

Our analysis of $\CCC M$ begins by establishing that the surjective algebra homomorphism $\ov{\sigma}_*\colon \CCC M\to \CCC \Lambda_*(M)$ induced by $\sigma_*$ is the semisimple quotient in good characteristic.  The nilpotence of the kernel is an immediate consequence of Theorem~\ref{radicalthm} and Proposition~\ref{isLImap}.

\begin{Prop}\label{nilpotentkernel}
The algebra homomorphism $\ov{\sigma}_*\colon \CCC M\to \CCC \Lambda_*(M)$ has nilpotent kernel.
\end{Prop}

The following explicit decomposition of $\CCC \Lambda_*(M)$ is a special case of a general result on inverse semigroup algebras from~\cite{mobius1,mobius2}.   The reader should consult~\cite{Cohn,Stanley} for the relevant background on Rota's theory of M\"obius inversion for posets.

\begin{Prop}\label{isomap}
There is an isomorphism \[\alpha\colon \CCC \Lambda_*(M)\to \bigoplus_{X\in \Lambda(M)} \CCC G_{X}\] defined on $\Lambda_*(M)$ by
\[\alpha(X,g) = \sum_{Y\leq X} \rho_Y(g).\]   The inverse is given on $g\in G_X$ by \[\alpha\inv (g) = \sum_{Y\leq X} (Y,\rho_Y(g))\mu(Y,X)\] where  $\mu$ is the M\"obius function of the lattice $\Lambda(M)$.
\end{Prop}
%\begin{proof}
%Since $\rho_X(g)=g$ for $g\in G_X$ by Proposition~\ref{definerestrict}, Rota's M\"obius inversion theorem easily yields that $\alpha$ is an isomorphism of vector spaces with inverse as advertised.  It suffices to show that $\alpha$ is a homomorphism.  Indeed,
%\begin{align*}
%\alpha(X_1,g_1)\alpha(X_2,g_2) &= \sum_{Y_1\leq X_1}\rho_{Y_1}(g_1)\cdot \sum_{Y_2\leq X_2}\rho_{Y_2}(g_2)\\ &= \sum_{Y\leq X_1\wedge X_2}\rho_Y (g_1g_2)\\ &= \alpha(X_1\wedge X_2,\rho_{X_1\wedge X_2}(g_1g_2))
%\end{align*}
%because $\CCC G_{Y_1}\cdot \CCC G_{Y_2}=0$ unless $Y_1=Y_2$.
%\end{proof}

Notice that the composition \[\alpha\ov{\sigma}_{*}\colon kM\to \bigoplus_{X\in \Lambda(M)} \CCC G_{X}\] is nothing more than the map induced by the morphisms \[\rho_X\colon M\to G_X\cup \{0\}\subseteq kG_X\] discussed earlier.  Recall that \[\rho_X(m) = \begin{cases}e_Xme_X & m\in M_X\\ 0 & \text{else.}\end{cases}\] Indeed, \[\alpha\ov{\sigma}_*(m)= \sum_{Y\leq \sigma(m)}\rho_Y(\rho_{\sigma(m)}(m)) = \sum_{Y\leq \sigma(m)}\rho_Y(m) = \sum_{X\in \Lambda(M)}\rho_X(m)\] where the last equality uses that $\rho_X(m)=0$ if $\sigma(m)\ngeq X$.

If $N$ is a monoid, let $\Irr (N)$ denote the set of (equivalence classes of) irreducible  representations of $N$ over $k$.  An immediate consequence of Propositions~\ref{nilpotentkernel}
and~\ref{isomap} is the following corollary.

\begin{Cor}\label{reptheoryofHsiaosgp}
Let $M$ be a rectangular monoid and let $k$ be a field. Then
\begin{equation*}
\Irr(M)=\Irr(\Lambda_*(M)) = \coprod_{X\in \Lambda(M)} \Irr(G_X).
\end{equation*}
If $\p\in \Irr(G_X)$, then it can be viewed as an irreducible representation of $M$ via the projection $\rho_X\colon M\to G_X\cup \{0\}\subseteq kG_X$.

If the characteristic of $k$ does not divide the order of any maximal subgroup $G_X$, then $\ov{\sigma}_*$ is the semisimple quotient.
\end{Cor}

The above results also imply that $k$ is a splitting field for $M$ if and only if it is a splitting field for each maximal subgroup of $M$.  (In fact, this is true for finite monoids in general.)

It is proved in~\cite{AMSV} that if $M$ is a finite monoid and the characteristic of $k$ does not divide the order of any of its maximal subgroups, then $kM$ is split basic if and only if $M$ is a rectangular monoid and the maximal subgroups of $M$ are abelian and split over $k$.

If $\p\in \Irr(G_X)$, sometimes it is useful to recall the precise form of the corresponding representation, of $M$, denoted also by $\p$. Namely, one has
\begin{equation*}
\p(m) = \begin{cases} \p(e_Xme_X) & \sigma(m)\geq X\\ 0 & \text{else.}\end{cases}
\end{equation*}

\section{The Rhodes radical squared and the quiver of an orthogroup}
In this section, we show how the computation of the quiver of a band can be reduced to the left regular case, handled in~\cite{Saliola}.  More generally, one can reduce the computation of the quiver of any orthogroup to that of a left regular band of groups, a case treated by the authors in~\cite{rrbg}.

Let $M$ be a monoid.  Let us define the \emph{Rhodes radical squared} $\rad^2(M)$ to be the congruence on $M$ associated to the mapping $M\to kM/\rad^2(kM)$.  Thus $(m,n)\in \rad^2(M)$ if and only if $m-n\in \rad^2(kM)$.  Notice that the kernel of the induced morphism $kM\to k[M/\rad^2(M)]$ is contained in $\rad^2(kM)$ and so $kM$ has the same quiver as $k[M/\rad^2(M)]$.

Recall that a band is called \emph{regular} if it satisfies the identity $xyxzx=xyzx$, cf.~\cite[Chapter 5]{Almeida:book}.  This terminology is unfortunate since all bands are von Neumann regular, but it is well-entrenched.  The variety of regular bands is the join of the varieties of left regular bands and right regular bands.  The next proposition implies that one can compute the quiver of any band if one can compute the quiver of a regular band.  This in turn will be reduced to the case of a left regular band, afterward.

\begin{Prop}\label{radsquaredDO}
Let $M$ be a rectangular monoid and let $e\in E(M)$.  Then $(emene,emne)\in \rad^2(M)$ for any $m,n\in M$.  In particular, if $B$ is a band, then $B/\rad^2(B)$ is a regular band.
\end{Prop}
\begin{proof}
Clearly $\sigma_*(eme)=\sigma_*(em)$ and $\sigma_*(ene)=\sigma_*(ne)$ and so $eme-em$ and $ene-ne$ belong to $\rad(kM)$.  But \[(eme-em)(ene-ne) = emene-emene-emene+emne=emne-emene.\]  This completes the proof.
\end{proof}

Every band $B$ has a least congruence $\sim$ such that $B/{\sim}$ is a regular band.  The quotient is called the \emph{maximal regular band image} of $B$.  One can define similarly the maximal left and right regular band images of $B$.  It follows from Proposition~\ref{radsquaredDO} that the quiver of a band coincides with that of its maximal regular band image. In particular, the quiver of a free band coincides with the quiver of the free regular band.  We next show that one can reduce the computation of the quiver of a band to the case of left regular bands and their duals, right regular bands. It is known that a regular band is a subdirect product of a left and right regular band and that a band is regular if and only if Green's relations $\R$ and $\eL$ are congruences.

An orthogroup is a left regular band of groups if and only if it does not contain a two-element right zero semigroup, i.e., idempotents $e\neq f$ with $ef=f$ and $fe=e$.  Right regular bands of groups are characterized dually.  In particular, a band is left regular if and only if it does not contain a two-element right zero semigroup. Left (right) regular bands of groups are exactly the orthogroups whose idempotents form a left (right) regular band. An orthogroup is a \emph{regular band of groups} if its idempotents form a regular band. A regular band of groups is a subdirect product of a left regular and a right regular band of groups.

The results of~\cite{rrbg} imply that if $M$ is an orthogroup and $X,Y\in \Lambda(M)$ are incomparable, then $\Ext_{kM}^1(U,V)=0$ for any simple $kG_X$-module $U$ and simple $kG_Y$-module $V$.

If $M$ is a monoid, let $M_{\ell}$ be the quotient of $M$ by the smallest congruence identifying each two-element right zero subsemigroup of $M$ and dually let $M_{r}$ be the quotient of $M$ by the smallest congruence identifying each two-element left zero subsemigroup of $M$, i.e.,  $M_r=((M^{op})_{\ell})^{op}$.  It follows from standard results of semigroup theory, cf.~\cite[Chapter~4]{qtheor}, that $M_{\ell}$ (respectively, $M_r$) contains no two-element right (respectively left) zero subsemigroup.  In particular, if $M$ is a band (respectively, an orthogroup), then $M_{\ell}$ is the maximal left regular band (respectively, left regular band of groups) image of $M$, and dually for $M_r$.  Notice that if $M$ is a rectangular monoid, then $\sigma_*\colon M\to \Lambda_*(M)$ identifies right zero and left zero semigroups and thus factors through the projections $M\to M_{\ell}$ and $M\to M_r$.  It follows that $M$, $M_{\ell}$ and $M_r$ all have the same simple modules (lifted from $\Lambda^*(M)$).

\begin{Thm}\label{duh}
Let $M$ be a rectangular monoid, let $k$ be a splitting field for $M$ and suppose that $X,Y\in \Lambda(M)$ are comparable, say $X\leq Y$.  Let $U,V$ be simple $kG_X$-, $kG_Y$-modules respectively.   Then the number of arrows from $U$ to $V$ in the quiver of $kM$ is the number of arrows from $U$ to $V$ in the quiver of $kM_{\ell}$ and the number of arrows from $V$ to $U$ is the number of arrows from $V$ to $U$ in the quiver of $kM_r$.
\end{Thm}
\begin{proof}
We show that if $d\colon M\to \Hom_k(U,V)$ is a derivation and $X\leq Y$, then it factors uniquely through $M_{\ell}$.  The remainder of the theorem is dual.   Define an equivalence relation on $M$ by $m\sim n$ if $\sigma_*(m)=\sigma_*(n)$ and $d(m)=d(n)$.  This is a congruence.  Indeed, if $m\sim n$, then $d(am) = ad(m)+d(a)m=ad(n)+d(a)n=d(an)$ and similarly $d(ma)=d(na)$.  Suppose that $\{e,f\}$ form a right zero semigroup.  Then $\sigma_*(e)=\sigma_*(f)$ and $d(e)=ed(e)+d(e)e$ and $d(f)=fd(f)+d(f)f$.  Thus if $\sigma(e)=\sigma(f)\geq Y$, then $d(e)=0=d(f)$ and if $\sigma(e)=\sigma(f)\ngeq X$, then also $d(e)=0=d(f)$.  If $Y>\sigma(e)=\sigma(f)\geq X$, then $d(e) = d(fe) = fd(e)+d(f)e=d(f)$.  Thus we see that $e\sim f$.  The result follows.
\end{proof}

Theorem~\ref{duh} reduces the computation of the quiver of any orthogroup to the case of a left regular band of groups and the computation of the quiver of any band to a left regular band.  The quiver of a left regular band was computed in~\cite{Saliola} and the quiver of a left regular band of groups was computed by the authors in~\cite{rrbg}.  However, the main result of this paper subsumes the case of orthogroups.

For the convenience of the reader, we recall the results of Saliola~\cite{Saliola} and the authors~\cite{rrbg}.
A classical definition from semigroup theory~\cite{Green} is that of the $\eL$-class $L_m $ and $\R$-class $R_m$ of an element $m$ in a monoid $M$:
\begin{align*}
L_m &= \{n\in M\mid Mn=Mm\};\\
R_m &= \{n\in M\mid nM=mM\}.
\end{align*}

\begin{Thm}[Saliola]\label{saliolamain}
Let $M$ be a left regular band and $k$ a field.  The vertex set of $Q(kM)$ is $\Lambda(M)$.  If $X,Y\in \Lambda(M)$, then the following holds.
\begin{enumerate}
\item If $X\nless Y$, then there are no arrows from $X$ to $Y$;
\item Otherwise, let $e_X,e_Y$ be idempotent generators of $X,Y$, respectively.  We may assume  $e_X<e_Y$. Set $N=e_YM_X=e_YM_Xe_Y$.  Let $L_X$ be the $\eL$-class of $e_X$ in $N$.  Define $\sim$ to be the least equivalence relation on $L_X$ such that:
    \begin{itemize}
 \item   $x\sim x'$ whenever $x,x'\in L_X$ and there exists $n\in N\setminus \{e_Y\}$ with $X<\sigma(n)$ and $nx=x, nx'=x'$.
    \end{itemize}
Then the number of arrows from $X$ to $Y$ is $|L_X/{\sim}|-1$.
\end{enumerate}
\end{Thm}

The corresponding result for left regular bands of groups (which encompasses the previous result) is:

\begin{Thm}[Margolis/Steinberg]\label{rrbgmain}
Let $M$ be a left regular band of groups and $k$ a splitting field for $M$ such that the characteristic of $k$ does not divide the order of any maximal subgroup of $M$.  Let $X,Y\in \Lambda(M)$ and let $U,V$ be simple modules for $kG_X$ and $kG_Y$ respectively.  Let $e_X,e_Y$ be the respective identities of $G_X,G_Y$.
\begin{enumerate}
\item If $X\nless Y$, then there are no arrows from $U$ to $V$;
\item Otherwise, we may assume $e_X<e_Y$. Set $N=e_YM_X=e_YM_Xe_Y$.  Let $L_X$ be the $\eL$-class of $e_X$ in $N$.  Define $\sim$ to be the least equivalence relation on $L_X$ such that:
    \begin{itemize}
 \item   $x\sim x'$ if $x,x'\in L_X$ are such that $e_Xx=e_Xx'$ and there exists $n\in E(N)\setminus \{e_Y\}$ with $X<\sigma(n)$ and $nx=x, nx'=x'$.
    \end{itemize}
Denote by $[x]$ the $\sim$-equivalence class of $x\in L_X$.
Let $V_{X,Y}$ be the subspace of $k[L_X/{\sim}]$ with basis the differences $[x]-[e_Xx]$ with $x\in L_X\setminus G_X$.  Then $V_{X,Y}$ is a $G_Y\times G_X$-module via the action \[(g,h)([x]-[e_Xx]) = [gxh\inv]-[ge_Xxh\inv].\]  The number of arrows from $U$ to $V$ in $Q(kM)$ is then the multiplicity of $V\otimes_k U^*$ as an irreducible constituent of $V_{X,Y}$.
\end{enumerate}
\end{Thm}

We remark that in~\cite{rrbg} the authors worked with right modules and right regular bands of groups.  Also we formulated the result in a slightly different way: we counted the multiplicity of $V$ as an irreducible constituent of the $kG_Y$-module $V_{X,Y}\otimes_{kG_X} U$.  But it is not difficult to see that the isomorphisms
\begin{align*}
\Hom_{k[G_Y\times G_X]}(V_{X,Y}, V\otimes_k U^*)&\cong \Hom_{kG_Y}(V_{X,Y}\otimes_{kG_X} U,V)\\& \cong \Hom_{kG_X}(U,\Hom_{kG_Y}(V_{X,Y},V))
\end{align*}
hold.

Theorems~\ref{duh},~\ref{saliolamain} and~\ref{rrbgmain} effectively compute the quiver of an arbitrary band or orthogroup.

\begin{Example}
It is shown in~\cite{Saliola} that the quiver of the free left regular band on $A$ has vertex set the subsets of $A$.  There are $|X\setminus Y|-1$ arrows from $X$ to $Y$ if $Y\subsetneq X\subseteq A$ and no arrows otherwise.  If $FB(A)$ is the free band on $A$, then $FB(A)_{\ell}$ is the free left regular band on $A$ and $FB(A)_r$ is the free right regular band on $A$.  It follows that the quiver of $FB(A)$ is as follows.  The vertex set is again the power set of $A$.  There are $|X\setminus Y|-1$ arrows from $X$ to $Y$ and from $Y$ to $X$ if $Y\subsetneq X\subseteq A$ and these are all the arrows.
One obtains the same quiver also for the free regular band on $A$ and, in fact, for any relatively free band monoid in a variety above regular bands.

Saliola~\cite{Saliola} showed (credited to K.~Brown) that the algebra of the free left regular band is hereditary, which means that its quiver (as can be seen from the description above) has no directed cycles and there are no quiver relations (i.e., the admissible ideal is trivial). Clearly, the algebra of the free (regular) band is not hereditary because our description shows that there are directed cycles in its quiver. An important problem would be to determine the quiver relations of the free (regular) band.  In a forthcoming paper~\cite{hereditarybands}, we describe how to compute the global dimension of a left regular band algebra using topological methods.  As a consequence, we show that a large number of left regular bands, including the free left regular band, have hereditary algebras.  On the other hand, we show that every split basic hereditary algebra is a left regular band algebra.
\end{Example}

\section{The quiver for special classes of rectangular monoids}\label{stateresults}
In this section, we state without proof a description of the quiver for several subclasses of rectangular monoids that have popped up in the literature.  The proofs will appear in the next section.

\subsection{Irreducible morphisms in the Karoubi envelope}
It turns out to be convenient to describe the quiver of certain classes of rectangular monoids in terms of the irreducible and almost irreducible morphisms of the Karoubi envelope. Let us recall the notion of an irreducible morphism in a category in the sense of Auslander and Reiten~\cite{assem,AuslanderReiten}.  If $\mathscr C$ is a category, then a morphism $f\in \mathscr C(c,d)$ is \emph{irreducible} if it is neither a split monomorphism, nor a split epimorphism, and whenever $f=gh$, then either $h$ is a split monomorphism or $g$ is a split epimorphism.  The set of irreducible morphisms $f\colon c\to d$ of $\mathscr C$ will be denoted $\Irrm_{\mathscr C}(c,d)$.

It will also be convenient to consider a new notion, that of an almost irreducible morphism.  Let us say that a morphism $f\colon c\to d$ of $\mathscr C$ is \emph{almost irreducible} if whenever $f=gh$, then either $h$ is a split monomorphism or $g$ is a split epimorphism.  In other words, we remove the restriction that $f$ itself not be a split monomorphism or split epimorphism.  We denote the set of almost irreducible morphisms $f\colon c\to d$ by $\mathsf{AIrr}_{\mathscr C}(c,d)$.

If $\mathscr C$ is a small category, we define the \emph{quiver} $Q(\mathscr C)$ of $\mathscr C$ to be the quiver with vertex set the isomorphism classes of objects of $C$ and arrows the irreducible morphisms between fixed representatives of the isomorphism classes.  This is in analogy with the case of a $k$-linear category $\mathscr A$ with local endomorphism algebras, where the vertex set of its quiver consists of isomorphism classes of objects and the number of arrows between two isomorphism classes of objects is the dimension of the $k$-vector space of irreducible morphisms between representative objects of the isomorphism class~\cite{assem,Gabrielbook}.

Let us describe the irreducible morphisms in the Karoubi envelope $\mathscr K(M)$ of a monoid $M$.

\begin{Lemma}\label{irredmorphisminKaroubi}
Let $M$ be a monoid and $\mathscr K(M)$ be its Karoubi envelope.  Let $e,f\in E(M)$.  Then $m\in \Irrm_{\mathscr K(M)}(e,f)$ if and only if the following conditions occur:
\begin{enumerate}
\item\label{irredmorphisminKaroubi.1} $Mm\subsetneq Me$;
\item\label{irredmorphisminKaroubi.2} $mM\subsetneq fM$;
\item\label{irredmorphisminKaroubi.3} $m=ab$ with $a\in fM$ and $b\in Me$ implies either $aM=fM$ or $Mb=Me$.
\end{enumerate}
The morphism $m\colon e\to f$ is almost irreducible if and only if (3) holds.
\end{Lemma}
\begin{proof}
In light of Proposition~\ref{splitmonosKaroubi}, it suffices to show that condition (3) is equivalent to being almost irreducible. Suppose first $m\colon e\to f$ is almost irreducible.  If $m=ab$ with $a\in fM$ and $b\in Me$, then $a\in \mathscr K(M)(1,f)$, $b\in \mathscr K(M)(e,1)$ and $m=ab$ in $\mathscr K(M)$.  Thus, since $m\colon e\to f$ is almost irreducible, it follows that $b$ is a split monomorphism or $a$ is a split epimorphism.  Proposition~\ref{splitmonosKaroubi} now yields (3).

Conversely, suppose that $m$ satisfies (3) and $b\colon e\to e'$ and $a\colon e'\to f$ are morphisms of $\mathscr K(M)$ such that $m=ab$.  Then by (3), either $aM=fM$ or $Mb=Me$.  But then Proposition~\ref{splitmonosKaroubi} implies either $b\colon e\to e'$ is a split monomorphism or $a\colon e'\to f$ is a split epimorphism.
\end{proof}

Condition (3) of Lemma~\ref{irredmorphisminKaroubi} can be simplified for finite monoids.  To do so, we shall need a fundamental property of finite monoids known as \emph{stability}.  See~\cite[Appendix]{qtheor} for the classical proof.  We reproduce the short proof from~\cite{transformations} for the convenience of the reader.

\begin{Prop}\label{stabilityprop}
Let $M$ be a finite monoid. Then $M$ is stable, that is,
\begin{align*}
MmnM = MmM&\iff mnM=mM\\ MnmM=MmM&\iff Mnm=Mm
\end{align*}
for all $m,n\in M$.
\end{Prop}
\begin{proof}
Clearly $mnM=mM$ implies $MmnM=MmM$.  For the converse, suppose that $umnv=m$ with $u,v\in M$.  Then $mM\subseteq umnM$ and $mnM\subseteq mM$, whence $|mM|\leq |umnM|\leq |mnM|\leq |mM|$.  It follows $mnM=mM$.  The other equivalence is dual.
\end{proof}

Recall that an element $m$ of a monoid $M$ is \emph{regular} if $m\in mMm$.  Equivalently, $m$ is regular if and only if $mM=eM$ for some $e\in E(M)$, if and only if $Mm=Mf$ for some $f\in E(M)$.  If $M$ is finite, $m$ is regular if and only if $MmM=MeM$ for some $e\in E(M)$. See~\cite[Appendix A]{qtheor} or~\cite{CP} for details.  Non-regular elements are termed \emph{null elements}.  If $A\subseteq M$, then we will denote by $\Null(A)$ the set of null elements of $A$.

\begin{Prop}\label{simplifyirred}
Let $M$ be a finite monoid and $e,f\in E(M)$.  Let $X=MeM$ and $Y=MfM$.  Then
\begin{align*}
\mathsf{AIrr}_{\mathscr K(M)}(e,f)&=fMe\setminus \nup Y\nup X \\ \Irrm_{\mathscr K(M)}(e,f)&=\Null(\AIrr_{\mathscr K(M)}(e,f)).
\end{align*}
In particular, $\AIrr_{\mathscr K(M)}(e,f)$ and $\Irrm_{\mathscr K(M)}(e,f)$ are $G_f\times G_e^{op}$-invariant subsets of $fMe$.
\end{Prop}
\begin{proof}
If $m\colon e\to f$ is almost irreducible and $m=st$ with $s,t\in M$, then $m=ab$ where $a=fs$ and $b=te$. Condition (3) of Lemma~\ref{irredmorphisminKaroubi} now implies $aM=fM$ or $Mb=Me$.  Thus we cannot have both $s\in \nup Y$ and $t\in \nup X$.  Conversely, suppose $m\notin \nup Y\nup X$ and $m=ab$ with $a\in fM$ and $b\in Me$.  Then $a\in Y$ or $b\in X$.  In the former case, we have $MfaM=MaM=MfM$ and so by stability $aM=faM=fM$.  Dually, if $b\in X$, then $Mb=Me$.  We conclude that $m\colon e\to f$ is almost irreducible.

Suppose now that $m\colon e\to f$ is irreducible.  Then it is almost irreducible. It remains to show that $m$ is null. Conditions \eqref{irredmorphisminKaroubi.1} and \eqref{irredmorphisminKaroubi.2} of Lemma~\ref{irredmorphisminKaroubi} imply $eM\neq mM$ and $Mf\neq Mm$.  Therefore, stability yields $m\in \nup Y\cap \nup X$.  But then, if $m$ were regular we would have $m=mnm$ for some $n\in M$ and so $m\in \nup Y\nup X$, contradicting that $m\in \AIrr_{\mathscr K(M)}(e,f)$.  Thus $m$ is null.  Conversely, suppose $m\in \Null(\AIrr_{\mathscr K(M)}(e,f))$.  Then since $m$ is not regular, conditions \eqref{irredmorphisminKaroubi.1} and \eqref{irredmorphisminKaroubi.2} of Lemma~\ref{irredmorphisminKaroubi} hold.  Thus $m\colon e\to f$ is irreducible.
\end{proof}

\begin{Rmk}\label{casesofAIrr}
We note here several easy consequences of Proposition~\ref{simplifyirred}.
 First observe that \[\mathsf{AIrr}_{\mathscr K(M)}(e,f)=fMe\setminus f\nup Y\nup Xe\] and that $f\nup Y\subseteq MYM\setminus Y$, $\nup Xe\subseteq MXM\setminus X$.  This is useful for simplifying computations.

 Next observe that if $e,f\in E(M)$ with $MeM$ and $MfM$ incomparable, then  $\Irrm_{\mathscr K(M)}(e,f)=\AIrr_{\mathscr K(M)}(e,f)$ since there are no regular elements in $fMe\setminus \nup Y\nup X$ in this case.

 If $MeM\subsetneq MfM$, then $\AIrr_{\mathscr K(M)}(e,f)$ consists of $\Irrm_{\mathscr K(M)}(e,f)$ together with the split monomorphisms from $e$ to $f$, i.e., those $m\in fMe$ with $Me=Mm$.

 Dually, if $MeM\supsetneq MfM$, then $\AIrr_{\mathscr K(M)}(e,f)$ consists of $\Irrm_{\mathscr K(M)}(e,f)$ together with the split epimorphisms from $e$ to $f$, i.e., those $m\in fMe$ with $fM=mM$.

 Finally, one has that $\AIrr_{\mathscr K(M)}(e,e)=\Irrm_{\mathscr K(M)}(e,e)\cup G_e$.
\end{Rmk}

In~\cite{Jtrivialpaper}, it is observed that for idempotent-generated $\J$-trivial monoids, quiver computations simplify.  The next proposition indicates the general result that can be used to simplify the computations.

\begin{Prop}\label{regulargenerated}
Let $M$ be a finite rectangular monoid generated by regular elements.  Let $e,f\in E(M)$.
\begin{enumerate}
\item\label{regulargenerated.1} If $MeM$ and $MfM$ are incomparable, then every element $m\in \AIrr_{\mathscr K(M)}(e,f)=\Irrm_{\mathscr K(M)}(e,f)$ can be factored as a split monomorphism followed by a split epimorphism, i.e.,  is of the form $ba$ with $Ma=Me$ and $bM=fM$.
\item\label{regulargenerated.2} If $MeM\subsetneq MfM$, then $\AIrr_{\mathscr K(M)}(e,f)$ consists of all split monomorphisms, i.e., elements $m\in fMe$ with $Me=Mm$.
\item\label{regulargenerated.3} If $MeM\supsetneq MfM$, then $\AIrr_{\mathscr K(M)}(e,f)$ consists of all split epimorphisms, i.e., elements $m\in fMe$ with $fM=mM$.
\item\label{regulargenerated.4}  $\AIrr_{\mathscr K(M)}(e,e)=G_e$.
\end{enumerate}
\end{Prop}
\begin{proof}
Let $A$ be a set of regular elements generating $M$.  Let $e,f\in E(M)$ and let $X$ and $Y$ be the conjugacy classes of $e$ and $f$ respectively. To prove \eqref{regulargenerated.1}, suppose that  $m\in \AIrr_{\mathscr K(M)}(e,f)$ and $m=a_1\cdots a_k$.  Assume $a_ib_ia_i=a_i$ with $b_i\in M$.  Since $m\in fMe$, it follows that $m\in \nup Y\cap \nup X$.  Thus there exists a least index $i$ such that $\sigma(a_i)\ngeq Y$.  Suppose that $\sigma(a_j)\ngeq X$ for some $j\geq i$.  Then $m=a_1\cdots a_i(b_ia_i\cdots a_j)a_{j+1}\cdots a_k$ and so belongs to $\nup Y\nup X$. This contradiction establishes that $\sigma(a_j)\geq X$ for $j\geq i$. Putting $b=fa_1\cdots a_{i-1}$ and $a=a_i\cdots a_ke$, we have $m=ba$ with $Me=Ma$ and $bM=fM$.

Next we observe that if $MeM\subseteq MfM$, then there are no null elements in $fMe\setminus \nup Y\nup X$.  Indeed, if $m\in fMe$ is a null element and $m=a_1\cdots a_k$ with the $a_i\in A$, then there must be some $a_j$ with $\sigma(a_j)\ngeq X$.  Thus if $a_jb_ja_j=a_j$ with $b_j\in M$, then $m = (a_1\cdots a_jb_j)(a_j\cdots a_k)\in \nup Y\nup X$.  Dually, if $MeM\supseteq MfM$, then there are no null elements in $fMe\setminus \nup Y\nup X$.  In light of Proposition~\ref{simplifyirred} and Remark~\ref{casesofAIrr}, \eqref{regulargenerated.2}--\eqref{regulargenerated.4} follow.
\end{proof}

For instance, if $M$ is $\mathscr J$-trivial and generated by idempotents, then in case Proposition~\ref{regulargenerated}\eqref{regulargenerated.1}, either $\AIrr_{\mathscr K(M)}(e,f)$ is empty or $\AIrr_{\mathscr K(M)}(e,f)=\{fe\}$.  In the remaining cases there is only one element in this set.

\subsection{Orthogroup and bands}
Let us begin with the case of orthogroups.  In the previous section, we have in principle reduced this case to that of left regular bands of groups, which we studied in our previous paper~\cite{rrbg}.  But let us give the direct answer here, which is perhaps computationally more efficient.

\begin{Thm}\label{orthogroupquiver}
Suppose that $M$ is an orthogroup, $k$ is a splitting field for $M$ and the characteristic of $k$ divides the order of no maximal subgroup of $M$.  Let $X,Y\in \Lambda(M)$ and suppose that $U\in \Irr(G_X)$, $V\in \Irr(G_Y)$.  Let $e_X,e_Y$ be the respective identities of $G_X$ and $G_Y$.  Then the following hold.
\begin{enumerate}
 \item If $X,Y$ are incomparable or $X=Y$, then there are no arrows from $U$ to $V$ in $Q(kM)$.
 \item If $X<Y$, then without loss of generality we may assume $e_X < e_Y$.  Let
 $\equiv$ be the smallest equivalence relation on $e_YL_{e_X}$ such that the following hold:
\begin{itemize}
\item [(a)] $e_Ymx\equiv e_Yme_Yx$ for $x\in L_{e_X}$ and $\sigma(m)\geq X$;
\item [(b)] $x\equiv x'$ whenever $x,x'\in e_YL_{e_X}$, $e_Xx=e_Xx'$ and there exists $e\in E(e_YMe_Y)$ with $X<\sigma(e)$, $e\neq e_Y$ and $ex=x, ex'=x'$.
\end{itemize}
The $\equiv$-class of $x$ will be denoted by $[x]$.

Let $V_{X,Y}$ be the subspace of $k[e_YL_{e_X}/{\equiv}]$ with basis all differences $[x]-[e_Xx]$ such that $x\in e_YL_{e_X}\setminus G_X$.  Then $V_{X,Y}$ is a $G_Y\times G_X$-module via the action \[(g,h)([x]-[e_Xx]) =[gxh\inv]-[ge_Xxh\inv].\]
The number of arrows from $U$ to $V$ is the multiplicity of $V\otimes_k U^*$ as an irreducible constituent of $V_{X,Y}$.
\item If $Y<X$, then the number of arrows from $U$ to $V$ coincides with the number of arrows from $V$ to $U$ in $kM^{op}$ and so can be computed as in the previous case.
\end{enumerate}
\end{Thm}
precisely
When $M$ is a regular band of groups, one has in the case $e_X<e_Y$ that \[e_Ymx=e_Ymxe_Y=e_Yme_Yxe_Y=e_Yme_Yx\] for $x\in L_{e_X}$.  Thus one can remove (a) from the definition of $\equiv$ without changing the equivalence relation so obtained.  It is then almost immediate to translate Theorem~\ref{orthogroupquiver} into Theorem~\ref{rrbgmain} and its dual.  Note that if $X<Y$, then $e_YL_{e_X}=\AIrr_{\mathscr K(M)}(e_X,e_Y)\setminus \Irrm_{\mathscr K(M)}(e_X,e_Y)$.

Let us specialize the above result to bands. In this case, all the maximal subgroups are trivial and so we obtain:

\begin{Cor}\label{bandquiver}
Suppose that $M$ is a band and $k$ is a field. Then the vertex set of $Q(kM)$ is the support lattice $\Lambda(M)$.  Let $e_X,e_Y$ be generators of $X,Y\in \Lambda(M)$.  Then the following hold:
\begin{enumerate}
 \item If $X,Y$ are incomparable or $X=Y$, there are no arrows from $X$ to $Y$ in $Q(kM)$.
 \item If $X<Y$, then without loss of generality we may suppose $e_X < e_Y$.  Let
 $\equiv$ be the smallest equivalence relation on $e_YL_{e_X}$ such that the following hold:
 \begin{itemize}
\item [(a)] $e_Ymx\equiv e_Yme_Yx$ for $x\in L_{e_X}$ and $\sigma(m)\geq X$;
\item [(b)] $x\equiv x'$ whenever $x,x'\in e_YL_{e_X}$ are such that there exists $e\in e_YMe_Y\setminus \{e_Y\}$ with $X<\sigma(e)$ and $ex=x, ex'=x'$.
\end{itemize}
The number of arrows from $X$ to $Y$ is $|e_YL_{e_X}/{\equiv}|-1$.
\item If $Y<X$, then the number of arrows from $X$ to $Y$ coincides with the number of arrows from $Y$ to $X$ in $Q(kM^{op})$ and so can be computed via the above.
\end{enumerate}
\end{Cor}

One can verify, as above, that if $M$ is a regular band, then one can drop (a) from the definition of $\equiv$ and the corollary reduces to Theorem~\ref{saliolamain} and its dual.

Numerous examples of quivers of left regular bands are computed in~\cite{Saliola}.  The authors carry out in~\cite{rrbg} a number of computations of quivers of left regular bands of groups, including the quiver of Hsiao's semigroup~\cite{Hsiao}.  In the previous section, we computed the quivers of (relatively) free bands. Thus we content ourselves with one last example of the quiver of a band.

\begin{Example}
Let $X$ be the semigroup with underlying set $\{a,b\}\times \{1,2\}$ and multiplication \[(x,y)(z,w) = (x,w).\]  It is an example of what is called a rectangular band.  Let $M=\{1,e\}\cup X$ where $1$ is the identity, $e$ is an idempotent and $e(x,y)=(a,y)$, $(x,y)e=(x,y)$ for $(x,y)\in X$.  Then $M$ is a band and $X$ is the minimal ideal of $M$.  It is easy to verify that $M$ is a regular band, but is neither left nor right regular. Sometimes $M\setminus \{1\}$ is called in the literature the \emph{singular square} semigroup.   We shall prove that $kM$ is a hereditary algebra isomorphic to the algebra of $3\times 3$ upper triangular matrices over $k$.

First we have that $\Lambda(M)=\{X,Y,Z\}$ with $Z=M$ and $Y=MeM$. The order is given by $Z>Y>X$. Take $e_X=(a,1)$.  So $L_{e_X} = \{(a,1),(b,1)\}$.  Thus $eL_{e_X}=\{(a,1)\}$ and so there are no arrows from $X$ to $Y$.

Next, we compute the number of arrows from $X$ to $Z$. The equivalence relation $\equiv$ identifies no elements of $L_{e_X}$ in this case and so there is one arrow from $X$ to $Z$.

There are no arrows between $Y$ and $Z$ since these ideals are generated by singleton conjugacy classes of idempotents.  There is one arrow from $Y$ to $X$.  This is because $R_{e_X}e=\{(a,1),(a,2)\}$ and the equivalence relation $\equiv$ identifies no elements.  Thus the quiver $Q(kM)$ is the Dynkin diagram $A_3$:
\[\xymatrix{Y\ar[r]&X\ar[r]&Z}.\]
The path algebra of this quiver is well known to be isomorphic to the algebra of $3\times 3$ upper triangular matrices over $k$.  The dimension of $kQ(M)$ is then $6$, which is exactly the dimension of $kM$.  Since $kM$ is a split basic algebra, it follows from Gabriel's theorem that $kM\cong kQ(M)$, as desired.

This example can be generalized as follows.  One can replace $\{a,b\}$ by any set $A$ containing $a$ and $\{1,2\}$ by any set containing $1$.  One then replaces $X$ by $A\times B$ where $A$ is given the left zero multiplication and $B$ the right zero multiplication to obtain a monoid $M(A;B)=\{1,e\}\cup X$.  The same computation as above shows that there are $|A|-1$ arrows from $X$ to $Z$ and $|B|-1$ arrows from $Y$ to $X$.  The path algebra then has dimension \[3+|A|-1+|B|-1+(|A|-1)(|B|-1) = 2+|A||B|=|M(A;B)|.\]  Thus $kM(A;B)\cong kQ(M(A;B))$ and hence is hereditary.  The monoids $M(A;B)$ are examples of projective finite monoids~\cite{qtheor}.  The authors can prove that all projective finite monoids have hereditary algebras (unpublished).
\end{Example}

\subsection{$\J$-trivial monoids and $\pv {DG}$}
Recall that a monoid $M$ is $\J$-trivial, meaning that Green's $\J$-relation is trivial, if $MmM=MnM$ implies $m=n$. The class of $\J$-trivial monoids plays a fundamental role in automata and language theory. The key reason is that ideals in the subword order in a free monoid are recognized by $\J$-trivial monoids and generate the variety of languages recognized by $\J$-trivial monoids~\cite{lothaire,SimonJtriv,Eilenberg,Pinbook}. Also, the class of $\J$-trivial monoids is generated as a variety of finite monoids by partially ordered monoids in which the identity is the smallest element~\cite{StraubingTherien}.  For example, the $0$-Hecke monoid $M(W)$ is partially ordered by the strong Bruhat order with the identity the smallest element~\cite{Norton,Fayers}.  The representation theory of $\J$-trivial monoids was intensively studied in~\cite{Jtrivialpaper}.  See also~\cite{Kiselmanstuff} where an important special case is considered.

A rectangular monoid belongs to the class $\pv{DG}$ if each conjugacy class of idempotents is a singleton. This is equivalent to saying that every regular $\D$-class is a group and this explains the notation $\pv{DG}$. Clearly, each $\J$-trivial monoid belongs to $\pv {DG}$.  We first compute the quiver of a monoid in $\pv{DG}$.  When we specialize to $\J$-trivial monoids, we recover the result of~\cite{Jtrivialpaper} on quivers of $\J$-trivial monoids, although our techniques are completely different and apply to the class of all rectangular monoids.

It is easily verified that $M$ belongs to $\pv{DG}$ if and only if $(xy)^{\omega}=(yx)^{\omega}$ for all $x,y\in M$~\cite{Almeida:book}.  It follows that the class $\pv{DG}$ is closed under direct product, submonoids and homomorphic images.  Clifford monoids belong to $\pv{DG}$, as do monoids with central idempotents. It is known that if $G$ is a finite group and $P(G)$ is the monoid of all subsets of $G$ under set multiplication, then $P(G) \in \pv{DG}$ if and only if every subgroup of $G$ is normal. For example, if $Q$ is the group of quaternions, then the power set monoid $P(Q)$ belongs to $\pv{DG}$.  We shall see that the representation theory of EI-categories~\cite{Luckbook,Webb,WebbEI} also reduces to the representation theory of monoids in $\pv{DG}$.

Let $M$ be a monoid.  If $m\in M$, let
\begin{align*}
\St_L(m)&=\{n \in M|nm=m\} \\
\St_R(m)&=\{n \in M|mn=m\}
\end{align*}
denote the left and right stabilizers of $m$, respectively. Both of these are submonoids of $M$. In particular, if $M\in \pv{DG}$, then they belong to $\pv{DG}$ and hence each have a unique idempotent in their minimal ideal, which we denote by $e_L(m)$ and $e_R(m)$, respectively.  Of course, $m\in e_L(m)Me_R(m)$.

We now describe the irreducible elements of the Karoubi envelope of a monoid in $\pv{DG}$.

\begin{Lemma}\label{cirreducible}
Let $M\in \pv{DG}$ and let $X,Y\in \Lambda(M)$.  As usual, $e_X,e_Y$ denote idempotent generators of $X,Y$, respectively. Then $m\in \Irrm_{\mathscr K(M)}(e_X,e_Y)$ if and only if:
\begin{enumerate}
\item $e_X=e_R(m)$;
\item $e_Y=e_L(m)$;
\item $m\notin \nup Y\nup X$;
\item $m$ is a null element.
\end{enumerate}
\end{Lemma}
\begin{proof}
Sufficiency is immediate from Proposition~\ref{simplifyirred}.  For necessity, the fact that $m\in \Irrm_{\mathscr K(M)}(e_X,e_Y)$ implies that $m\in \Null(e_YMe_X\setminus \nup Y\nup X)$ by Proposition~\ref{simplifyirred}.  Thus $\sigma(m)<X$ and $\sigma(m)<Y$ because $m$ is null.  Suppose first that $e_Y\neq e_L(m)$.  As $e_Y\in \St_L(m)$, it follows $\sigma(e_L(m))< Y$. Thus $m=e_L(m)m\in \nup Y\nup X$, a contradiction.  We conclude $e_Y=e_L(m)$.  The case $e_X\neq e_R(m)$ is dual.
\end{proof}

For a monoid $M$ in $\pv{DG}$, the set $\Lambda(M)$ is in canonical bijection with $E(M)$, which allows us to ease the notation.

\begin{Thm}\label{quiverDG}
Let $M\in \pv{DG}$ and suppose that $k$ is a splitting field for each maximal subgroup of $M$ and that the characteristic of $k$ divides the order of no maximal subgroup of $M$. The vertex set of the quiver $Q(kM)$ is the disjoint union $\coprod_{e\in E(M)} \Irr(G_e)$.  If $U\in \Irr(G_e)$, $V\in \Irr(G_f)$, then the number of arrows from $U$ to $V$ in the quiver of $kM$ is the multiplicity of $V\otimes_k U^*$ as an irreducible constituent of the $G_f\times G_e$-module $k[\Irrm_{\mathscr K(M)}(e,f)]$.
\end{Thm}

In the case of a $\J$-trivial monoid $M$, the maximal subgroups are trivial and so the vertex set of $Q(kM)$ is $E(M)$ and the number of arrows from $e$ to $f$ is $|\Irrm_{\mathscr K(M)}(e,f)|$.  Since no two idempotents of $M$ are conjugate, Proposition~\ref{Dequiv} yields that that no two objects of $\mathscr K(M)$ are isomorphic.  Thus Theorem~\ref{quiverDG} reduces to the following result.

\begin{Cor}\label{Jtrivialquiver}
Let $M$ be a $\J$-trivial monoid and $k$ a field.  Then one has $Q(kM)=Q(\mathscr K(M))$.  That is, the vertex set of $Q(kM)$ is $E(M)$ and the set of arrows from $e$ to $f$ is $\Irrm_{\mathscr K(M)}(e,f)$.
\end{Cor}

We remark that if $M$ is $\J$-trivial, then Lemma~\ref{cirreducible} shows that $m\in\Irrm_{\mathscr K(M)}(e,f)$ if and only if $m$ is $c$-irreducible in the sense of~\cite{Jtrivialpaper} and satisfies $e_R(m)=e$ and $e_L(m)=f$.  Thus we have recovered the result of~\cite{Jtrivialpaper} on quivers of $\J$-trivial monoids.  Detailed examples of quivers of $\J$-trivial monoids are given in~\cite{Jtrivialpaper}.

\begin{Example}
Let $G$ be a group and $X$ a $G\times G$-set.  Assume $k$ is a splitting field for $G$ and the characteristic of $k$ does not divide the order of $G$.  Define a monoid $M=G\cup X\cup \{z\}$ where $z$ is a multiplicative zero, $X^2=\{z\}$ and $G$ is the group of units.  If $g\in G$ and $x\in X$, one has $gx=(g,1)x$ and $xg=(1,g\inv)x$.  Then $M$ belongs to $\pv{DG}$ as its idempotents are just $1$ and $z$, which are central.  One easily verifies that $X=\Irrm_{\mathscr K(M)}(1,1)$ and that there are no other irreducible morphisms in $\mathscr K(M)$.

Let $V_1,\ldots, V_s$ be the simple $kG$-modules.  Then the vertex set of $Q(kM)$ consists of $V_1,\ldots, V_s$ together with the trivial $kM$-module $k$.  The vertex $k$ is isolated.  The number of arrows from $V_i$ to $V_j$ is the multiplicity of $V_j\otimes_k V_i^*$ as an irreducible constituent of the $G\times G$-module $kX$.  Thus computing the quiver of $kM$ is equivalent to decomposing $kX$ into simple $G\times G$-modules.

For concreteness, let us take $X=G$ via the usual action $(h,k)g= hgk\inv$.  Then the stabilizer of $1$ is the diagonal subgroup $\Delta(G)$.  Thus Frobenius reciprocity shows that the multiplicity of $V_j\otimes_k V_i^*$ in the $G\times G$-module $kG$ is the same as the multiplicity of the trivial representation of $G$ in $V_j\otimes_k V_i^*$ (viewed as a $G$-module in the usual way).  But $V_j\otimes_k V_i^*\cong \Hom_k(V_i,V_j)$ where $G$ acts on $f\colon V_i\to V_j$ by $(gf)(v) = gf(g\inv v)$ for $g\in G$ and $v\in V$.  The subspace of $G$-invariants is exactly $\Hom_{kG}(V_i,V_j)$.  Thus, by Schur's lemma and the fact that $k$ is a splitting field, we conclude that the multiplicity is $1$ if $i=j$ and is $0$, otherwise.  In other words, there is a loop at each $V_i\in \Irr(G)$ and no other arrows.

On the other hand, if we take $X=G\times G$, then $kX$ is the regular $G\times G$-module and so the multiplicity of $V_j\otimes_k V_i^*$ in $kX$ is $\dim V_j\cdot \dim V_i$.  Thus there are $\dim V_j\cdot \dim V_i$ arrows from $V_i$ to $V_j$ for all $i,j$.
\end{Example}

\subsubsection{EI-categories}
We show next that the representation theory of finite EI-categories is subsumed by the representation theory of monoids in $\pv{DG}$.   The representation theory of categories is a classical subject~\cite{ringoids,Webb} that encompasses the theory of quiver representations.  A (finite dimensional) representation of a category $\mathscr C$ over a field $k$ is a functor from $\mathscr C$ to the category of (finite dimensional) $k$-vector spaces.  For instance, if $Q$ is a quiver, then a representation of $Q$, cf.~\cite{assem}, is precisely a representation of the free category on $Q$ (see~\cite{Mac-CWM} for the definition of free category).

From now on suppose that $\mathscr C$ is a finite category.  We shall denote by $\mathscr C^0, \mathscr C^1$ the sets of objects and arrows (or morphisms) of $\mathscr C$, respectively.  Given a field $k$, one can associate to $\mathscr C$ the \emph{category algebra} $k\mathscr C$ defined as follows~\cite{ringoids,Webb}.  One takes as a basis the set of arrows of $\mathscr C$ and extends the composition in $\mathscr C$ linearly, where undefined products are taken as $0$.  This algebra is unital with identity $\sum_{c\in \mathscr C^0}1_c$.  Moreover, the category of (finitely generated) $k\mathscr C$-modules is equivalent to the category of (finite dimensional) representations of $\mathscr C$~\cite{ringoids,Webb}.  Furthermore, if $\mathscr C$ and $\mathscr D$ are equivalent categories, then $k\mathscr C$ and $k\mathscr D$ are Morita equivalent~\cite{ringoids,Webb}.  Thus we may assume for the purposes of representation theory that the category $\mathscr C$ is \emph{skeletal}, meaning that distinct objects of $\mathscr C$ are not isomorphic.

An \emph{EI-category} is a category $\mathscr C$ in which every endomorphism is an isomorphism, that is, each endomorphism monoid $\mathscr C(c,c)$ is a group.  In the special case that each endomorphism monoid is trivial, we shall call $\mathscr C$ \emph{locally trivial}. For example, the free category on an acyclic quiver is a locally trivial category.  If $P$ is a poset, then there is an associated category $\mathscr P$ with $\mathscr P^0=P$ and with exactly one arrow from $p$ to $q$ if $p\leq q$, and no arrows otherwise.  This category is clearly locally trivial.  The category algebra of $\mathscr P$ is then the incidence algebra~\cite{Stanley} of the poset $P$ when $P$ is finite.

As another example, let $\mathscr I_n$ be the category whose objects are the sets $[m]=\{1,\ldots, m\}$ with $0\leq m\leq n$ (so $[0]=\emptyset$) and whose arrows are the one-to-one mappings.  Then $\mathscr I_n$ is a skeletal EI-category, the endomorphism monoid of $[m]$ being the symmetric group $S_m$. EI-categories play a role in algebraic $K$-theory~\cite{Luckbook} and have received quite some attention in the literature, see~\cite{Webb,WebbEI} and the references therein.

In an EI-category $\mathscr C$, every split monomorphism and split epimorphism is an isomorphism.  Thus a morphism $f\colon c\to c'$ is irreducible if and only if it is not an isomorphism and $f=gh$ implies $g$ or $h$ is an isomorphism.

To any finite category $\mathscr C$, one can associated a finite semigroup $S(\mathscr C)=\mathscr C^1\cup \{z\}$ where $z$ is a multiplicative zero and one extends the product in $\mathscr C$ by making all undefined products equal to $z$.  Observe that $k\mathscr C=kS(\mathscr C)/kz$ and hence, since $z$ is a central idempotent, $kS(\mathscr C)\cong k\mathscr C\times k$ (and so in particular is unital).  Also, since $k\mathscr C$ is unital, if $M(\mathscr C)$ is the monoid obtained by adjoining an identity $1$ to $S(\mathscr C)$, then $kM(\mathscr C)\cong k\mathscr C\times k\times k$.  Thus the simple modules for $k\mathscr C$ are obtained from those of $M(\mathscr C)$ by throwing away the simple modules associated to $z$ and $1$ and the quiver of $k\mathscr C$ is obtained from that of $kM(\mathscr C)$, by removing the vertices corresponding to these simple modules (which are isolated).  We shall refer to $z,1$ as the trivial idempotents of $M(\mathscr C)$.

Notice that $\mathscr C$ is isomorphic to the subcategory of $\mathscr K(M(\mathscr C))$ with objects the identities $1_c$ with $c\in \mathscr C^0$ and morphism sets $\mathscr K(M(\mathscr C))(1_c,1_d)\setminus \{z\}$.  Since $z$ is a regular element, it follows that $\Irrm_{\mathscr K(M(\mathscr C))}(1_c,1_d) = \Irrm_{\mathscr C}(c,d)$ by Proposition~\ref{simplifyirred}.

Suppose now that $\mathscr C$ is a skeletal EI-category.  Then the non-trivial idempotents of $M(\mathscr C)$ are precisely the identities $1_c$ with $c\in \mathscr C^0$.  Thus $\mathscr C$ is isomorphic to the subcategory obtained from $\mathscr K(M(\mathscr C))$ by removing the objects $1,z$ and all arrows $z\in \mathscr K(M(\mathscr C))(1_c,1_d)$.  In particular, $Q(\mathscr C)$ is the subquiver of $Q(\mathscr K(M(\mathscr C))$ obtained by removing the objects $1,z$ (which are isolated).  The fact that $\mathscr C$ is skeletal then implies that $\mathscr K(M(\mathscr C))$ is skeletal and hence $M(\mathscr C)$ belongs to $\pv{DG}$.  The maximal subgroup at $1_c$ is just the group $\mathscr C(c,c)$.

We can now describe the quiver of $k\mathscr C$ for a finite EI-category $\mathscr C$.  The result is immediate from Theorem~\ref{quiverDG} and the above discussion.

\begin{Thm}\label{EIquiver}
Let $\mathscr C$ be a finite skeletal EI-category and $k$ a field.  Assume that, for each $c\in \mathscr C^0$, the characteristic of $k$ does not divide $|\mathscr C(c,c)|$ and that $k$ is a splitting field for $\mathscr C(c,c)$.  Then the following hold:
\begin{itemize}
\item the vertex set of $Q(k\mathscr C)$ is $\coprod_{c\in \mathscr C^0}\Irr(\mathscr C(c,c))$;
\item if $U\in \Irr(\mathscr C(c,c))$ and $V\in \Irr(\mathscr C(c',c'))$, then the number of arrows from $U$ to $V$ is the multiplicity of $V\otimes_k U^*$ as an irreducible constituent in the $\mathscr C(c',c')\times \mathscr C(c,c)$-module $k[\Irrm_{\mathscr C}(c,c')]$.
\end{itemize}
\end{Thm}

Most likely, Theorem~\ref{EIquiver} is known in some form since one can compute the primitive idempotents of $k\mathscr C$ from the primitive idempotents for the group algebras $k\mathscr C(c,c)$.

Notice that if $\mathscr C$ is a skeletal locally trivial category, then $M(\mathscr C)$ is $\J$-trivial and so Theorem~\ref{EIquiver} admits the following simplification.

\begin{Thm}
Let $\mathscr C$ be a finite skeletal locally trivial category and $k$ a field.  Then $Q(k\mathscr C)=Q(\mathscr C)$. That is, $Q(k\mathscr C)$ has vertex set $\mathscr C^0$ and arrow set $\Irrm_{\mathscr C}(c,d)$ from $c$ to $d$.
\end{Thm}

For example, if $Q$ is a finite acyclic quiver, then the irreducible morphisms of the free category on $Q$ are the arrows of $Q$ and so the quiver of $kQ$ is $Q$, as expected.  Similarly, if $P$ is a poset, then the irreducible morphisms of the associated category $\mathscr P$ are the arrows corresponding to covers and hence the quiver of the incidence algebra of $P$ is the Hasse diagram of the poset, as is well known.

\begin{Example}[Brimacombe]
In this example we compute the quiver of $\mathbb C\mathscr I_n$ using Theorem~\ref{EIquiver}.  This was originally computed by the second author's student, B.~Brimacombe, via direct computation with primitive idempotents as part of her Ph.~D.\ thesis (in progress).

One has $\mathscr I_n([m],[m])=S_m$ and so the vertices of $Q(\mathbb C\mathscr I_n)$ can be identified with Young diagrams with at most $n$ boxes.  It is easy to see that a morphism $f\colon [m]\to [k]$ is irreducible if and only if $k=m+1$ (and so $m<n$).  One can identify $\mathscr I_n([m],[m+1])$ with $S_{m+1}$ by associating an injective mapping $f\colon [m]\to [m+1]$ with its unique extension to a permutation of $[m+1]$.  Identifying $S_m$ with the permutations of $[m+1]$ that fix $m+1$, we see that $\mathrm{Irr}([m],[m+1])$ can be identified as an $S_{m+1}\times S_{m}$-set with $S_{m+1}$ equipped with the action $(g,h)k = gkh\inv$.  Notice that the stabilizer of the identity permutation is the diagonally embedded copy of $S_m$.  Thus if $S^{\lambda}\in \Irr(S_m)$ and $S^{\mu}\in \Irr(S_{m+1})$ are Specht modules associated to partitions $\lambda\vdash m$ and $\mu\vdash m+1$, then by Frobenius reciprocity, the multiplicity of $S^{\mu}\otimes_{\mathbb C} (S^{\lambda})^*\cong \Hom_{\mathbb C}(S^\lambda,S^\mu)$ in $\mathbb CS_{m+1}$ is the multiplicity of the trivial representation of $S_m$ in $\Hom_{\mathbb C}(S^\lambda,S^\mu)$ as an $S_{m+1}\times S_m$-module (where $S^\mu$ is viewed as an $S_m$-module via restriction). This in turn is the dimension of $\Hom_{\mathbb CS_m}(S^{\lambda},S^{\mu})$.  By the branching rule for symmetric groups~\cite[Theorem~2.4.3]{jameskerber}, this multiplicity is $1$ if $\lambda$ can be obtained from $\mu$ by removing a single box and $0$, otherwise.  Thus the quiver of $\mathbb C\mathscr I_n$ is the Hasse diagram of Young's lattice~\cite[Chapter 7]{Stanley2}, truncated after rank $n$.
\end{Example}

\subsection{$\R$-trivial monoids}
Recall that a monoid $M$ is \emph{$\R$-trivial} if $mM=nM$ implies $m=n$, that is, Green's relation $\R$ is trivial.  Notice that $\J$-trivial monoids are precisely the $\R$-trivial monoids whose conjugacy classes of idempotents are singletons, whereas left regular bands are the von Neumann regular $\R$-trivial monoids.  Schocker gave an alternative axiomatization of $\R$-trivial monoids in~\cite{SchockerRtrivial}; cf.~\cite{BergeronSaliola}.  A complete set of orthogonal primitive idempotents for the algebra of an $\R$-trivial monoid was computed in~\cite{BergeronSaliola}.  However, the quiver and Cartan matrix were not computed, nor was a semigroup theoretic model of the projective indecomposable modules provided (as was done for left regular bands in~\cite{Saliola} and $\J$-trivial monoids in~\cite{Jtrivialpaper}).  We do this here, independently of the results of~\cite{BergeronSaliola}.  Our proof does not even use primitive idempotents, and so is semigroup theoretic.

An important example of an $\R$-trivial monoid is the monoid $E_n$ of all mappings $f\colon \{1,\ldots, n\}\to \{1,\ldots,n\}$, acting on the \textbf{right}, satisfying $if\leq i$ for all $i\in \{1,\ldots,n\}$.  Such mappings are called \emph{extensive} in the French school of finite semigroup theory, whence the notation~\cite{Pinbook}.  The Cayley theorem for $\R$-trivial monoids says that if $M$ is a finite $\R$-trivial monoid with $n$ elements, then $M$ embeds in $E_n$~\cite[Chapter 4]{Pinbook}.

The class of $\R$-trivial monoids is closed under semidirect product.  It is an important and non-trivial fact that the class of finite $\R$-trivial monoids is the smallest class of monoids containing the two-element lattice $\pv 2$ and which is closed under taking semidirect products, submonoids and homomorphic images~\cite{Eilenberg,qtheor,Stiffler}.

We now describe the quiver of an $\R$-trivial monoid.
%
%\begin{Thm}\label{Rtrivialquiver}
%Let $M$ be an $\R$-trivial monoid and $k$ a field.  The vertex set of $Q(kM)$ is $\Lambda(M)$.  Let $X,Y\in \Lambda(M)$ and let $e_X,e_Y$ be idempotent generators of $X$ and $Y$, respectively.  Then the number of arrows from $X$ to $Y$ in $Q(kM)$ are determined as follows.
%\begin{enumerate}
%\item If $X$ and $Y$ are incomparable, let $Z_{X,Y}=e_YMe_X\setminus \nup Y\nup X$. Let $\equiv$ be the least equivalence relation on $Z_{X,Y}$ such that $e_Ymne_X\equiv e_Yme_X$ whenever $\sigma(m)\ngeq Y$ and $\sigma(n)\geq X$.  Then the number of arrows from $X$ to $Y$ is $|Z_{X,Y}/{\equiv}|$.
%\item If $X=Y$, let $Z_X=e_XMe_X\setminus \left((\nup X)^2\cup \{e_X\}\right)$.  Define $\equiv$ to be the least equivalence relation on $Z_X$ such that $e_Xmne_X\equiv e_Xme_X$ whenever $\sigma(m)\ngeq X$ and $\sigma(n)\geq X$.  Then the number of arrows from $X$ to $Y$ is $|Z_X/{\equiv}|$.
%\item If $X<Y$, then we may that assume $e_X<e_Y$.  Set $Z_{X,Y}= e_YMe_X\setminus \nup Y\nup X$ and denote by $\equiv$ the least equivalence relation on $Z_{X,Y}$ such that $e_Ymne_X\equiv e_Yme_Yne_X\equiv e_Yme_X$ whenever $\sigma(m)\ngeq Y$ and $\sigma(n)\geq X$.  Then the number of arrows from $X$ to $Y$ is $|Z_{X,Y}/{\equiv}|-1$.
%\item If $Y<X$, then we may assume $e_Y<e_X$.  Let $Z_{X,Y}=e_YMe_X\setminus \nup Y\nup X$.  Let $\equiv$ be the least equivalence relation on $Z_{X,Y}$ such that $e_Ymne_X\equiv e_Yme_X$ whenever $\sigma(m)\ngeq Y$ and $\sigma(n)\geq X$.  Then the number of arrows from $X$ to $Y$ is $|Z_{X,Y}/{\equiv}|-1$.
%\end{enumerate}
%\end{Thm}

\begin{Thm}\label{Rtrivialquiver}
Let $M$ be an $\R$-trivial monoid and $k$ a field.  The vertex set of $Q(kM)$ is $\Lambda(M)$.  Let $X,Y\in \Lambda(M)$ and let $e_X,e_Y$ be idempotent generators of $X$ and $Y$, respectively.  Then the number of arrows from $X$ to $Y$ in $Q(kM)$ is determined as follows.  Let $\equiv$ be the least equivalence relation on $e_YMe_X$ such that:
\begin{itemize}
\item [(a)]  $e_Ymne_X\equiv e_Yme_X$ whenever $\sigma(m)\ngeq Y$ and $\sigma(n)\geq X$;
\item [(b)] $z\equiv z'$ for all $z,z'\in \nup Y\nup X$.
\end{itemize}
Let $Z_{X,Y}$ be the complement of the class of $\nup Y\nup X$ in $e_YMe_X/{\equiv}$
\begin{enumerate}
\item If $X$ and $Y$ are incomparable, then the number of arrows from $X$ to $Y$ is $|Z_{X,Y}|$.
\item In all other cases, there are $|Z_{X,Y}|-1$ arrows from $X$ to $Y$.
\end{enumerate}
\end{Thm}

We remark, that one can replace (a) by the condition $e_Ymn\equiv e_Yme_X$ whenever $\sigma(m)\ngeq Y$ and $Mn=Me_X$ because $Mne_X=Me_X$ if and only if $\sigma(n)\geq X$.
Notice that the equivalence classes of $Z_{X,Y}$ are precisely the classes of elements $\AIrr_{\mathscr K(M)}(e_X,e_Y)$ which do not get identified with elements of $\nup Y\nup X$.

It is a routine, but tedious, exercise to deduce Theorem~\ref{saliolamain} and Corollary~\ref{Jtrivialquiver} from this result.

We have not yet been able to compute the quiver for $E_n$ with $n\geq 4$ because $|E_n|=n!$ and so performing hand computations is prohibitive.

\subsubsection{Projective indecomposable modules and Cartan invariants}
We end this section by giving a semigroup theoretic description of the projective indecomposable modules of the algebra of an $\R$-trivial monoid and computing the Cartan invariants. Our approach is inspired by that of~\cite{Jtrivialpaper} for $\J$-trivial monoids, however we avoid any use of primitive idempotents and so our proof can be viewed as semigroup theoretic.

We recall the notion of the $\til {\eL}$-relation and associated preorder on a monoid $M$~\cite{Ltilde}.  One defines $m\leq_{\til {\eL}}n$ if, for all $e\in E(M)$, one has \[ne=n\implies me=m.\]   For example, if $Mm\subseteq Mn$, then clearly $m\leq_{\til {\eL}}n$.   If $n$ is regular, then one easily verifies that the converse holds. Indeed, if $nn'n=n$, then $n'n\in E(M)$ and so $mn'n=m$.  Thus $Mm\subseteq Mn$. Thus when restricted to regular elements, $\til{\eL}$ coincides with the classical Green's $\eL$-relation. In particular, if $e,f$ are idempotents then $e\mathrel{\til {\eL}} f$ if and only if $ef=e$ and $fe=f$.

Notice that $mn\leq_{\til {\eL}}n$ for all $m,n\in M$.   Let us write $m\mathrel{\til{\eL}} n$ if $m\leq_{\til {\eL}}n$ and $n\leq_{\til {\eL}}m$.  Set $\til L_m$ to be the $\til{\eL}$-equivalence class of $M$.

\begin{Prop}\label{stabilizerstuff}
If $M$ is a finite $\R$-trivial monoid, then the idempotents of $\til L_m$ are the elements of the minimal ideal of $\St_R(m)$, which moreover is an $\eL$-class of $M$.
\end{Prop}
\begin{proof}
Let $e$ be an idempotent of the minimal ideal of $\St_R(m)$.  We claim $\St_R(m)=\St_R(e)$. Since $\St_R(m)$ is $\R$-trivial and $e$ belongs to its minimal ideal, it follows that $en=e$ for all $n\in \St_R(m)$ by stability (Proposition~\ref{stabilityprop}).  Conversely, if $n\in \St_R(e)$, then $mn=men=me=m$.  It follows that all elements of $\til L_m$ have the same right stabilizer.  It also follows that $e\mathrel{\til {\eL}} m$.  Since idempotents are $\til {\eL}$-equivalent if and only if they are $\eL$-equivalent, it follows that $E(\til L_m) = L_e$, which is moreover the minimal ideal of $\St_R(e)=\St_R(m)$.
\end{proof}

Fix now a finite $\R$-trivial monoid $M$ and an idempotent generator $e_X$ of each $X\in \Lambda(X)$.  Let $L_X$ be the $\eL$-class of $e_X$. Then $L_X=E(\til L_{e_X})$ and hence $\til L_{e_X}$ depends only on $X$; we therefore denote it by $\til L_X$.   Let us put
\begin{equation}\label{tilLnotation}
\til L(X)_<=ML_X\setminus \til L_X.
\end{equation}
Then $ML_X$ and $\til L(X)_<$ are left ideals of $M$.

We now prove the main result of this subsection, generalizing the results of Saliola for left regular bands~\cite{Saliola} and of~\cite{Jtrivialpaper} for $\J$-trivial monoids.

\begin{Thm}\label{projindec}
Let $M$ be an $\R$-trivial monoid, $k$ a field and $X\in \Lambda(M)$.  Denote by $S_X$ the simple module associated to $X$ and by $P_X$ its projective cover.  We retain the notation of \eqref{tilLnotation}.  Then $P_X\cong kML_X/k\til L(X)_<$.

More precisely, $P_X$ is isomorphic to the $kM$-module $k\til L_X$ with basis $\til L_X$ and with action defined on the basis by putting
\begin{equation}\label{schutzvar}
m\cdot n = \begin{cases} mn & \text{if}\ mn\in \til L_X\\ 0 & \text{else.}\end{cases}
\end{equation}
for $m\in M$ and $n\in \til L_X$.
\end{Thm}
\begin{proof}
It is clear that $kML_X/k\til L(X)_<\cong k\til L_X$.  The key step to prove the theorem is to establish that $k\til L_X/\rad(k\til L_X)\cong S_X$.  Let us assume this for the moment and complete the proof, which is a counting argument.  Indeed, it will then follow that $P_X$ is the projective cover of $k\til L_X$ and hence maps onto it~\cite{assem,benson}. In particular, $|\til L_X|=\dim k\til L_X\leq \dim P_X$.  But Proposition~\ref{stabilizerstuff} implies that $\{\til L_X\mid X\in \Lambda(X)\}$ is a partition of $M$.  Thus we have
\begin{equation*}
0\leq \sum_{X\in \Lambda(X)} (\dim P_X-|\til L_X|)=\dim kM-|M|=0.
\end{equation*}
and hence $\dim P_X=|\til L_X|$ for all $X\in \Lambda(X)$.  It follows that $P_X\cong k\til L_X$ for all $X$, as required.

There is clearly a surjective homomorphism $\theta\colon k\til L_X\to S_X\cong k$ sending each element of $L_X$ to $1$ and each element of $\til L_X\setminus L_X$ to $0$.  We need that $\ker\theta\subseteq \rad(k\til L_X)$. Note that $\ker \theta$ is spanned by the elements $e-f$ with $e,f\in L_X$ and by the elements $x\in \til L_X\setminus L_X$.  Now in the first case, $e-f\in \rad(kM)$ as $\sigma(e)=X=\sigma(f)$ and so $e-f=(e-f)e_X\in \rad(\til L_X)$.  If $x\in \til L_X\setminus L_X$, then since $xe_X=x$, we have $x^{\omega}e_X=x^{\omega}$.  But if $e_Xx^{\omega}=e_X$, then we would have $x^{\omega}\in L_X$ and hence $x\in L_X$.  We conclude that $x^{\omega}\notin \til L_X$.  Now $x-x^{\omega}\in \rad(kM)$ because $\sigma(x)=\sigma(x^{\omega})$.  In the module $k\til L_X$ (with action \eqref{schutzvar}) one has $(x-x^{\omega})e_X =x$ and so $x\in \rad(k\til L_x)$.  This completes the proof that $\ker \theta \subseteq \rad(k\til L_X)$ and so $k\til L_X/\rad(k\til L_X)\cong S_X$, thereby proving the theorem.
\end{proof}

The action in \eqref{schutzvar} is a variation on the classical Sch\"utzenberger representation~\cite{CP,qtheor,Arbib}, and has been considered before in the literature, eg.~\cite{Vickypaper}.

Next we compute the Cartan matrix.  We could use the results of~\cite{mobius2} to do this, but we instead give here a direct argument.

\begin{Thm}\label{Cartan}
Fix, for each $X\in \Lambda(M)$, an idempotent generator $e_X$.
Define, for $X,Y\in \Lambda(M)$, \[m_{X,Y} = |\{n\in \til L_Y\mid e_Xn=n\}|.\] Let $c_{X,Y} = \dim \Hom_{kM}(P_X,P_Y)$ so that the Cartan matrix is $C=(c_{X,Y})$.  Then \[c_{X,Y} = \sum_{Z\leq X} m_{Z,Y}\cdot \mu(Z,X)\] where $\mu$ is the M\"obius function of the lattice $\Lambda(M)$.
\end{Thm}
\begin{proof}
It is well known~\cite{benson} that $\dim \Hom_{kM}(P_X,P_Y)$ is the multiplicity of the simple module $S_X=P_X/\rad(P_X)$ as a composition factor in $P_Y$.  The character $\theta_Y$ of $P_Y$ is the sum of the characters of its composition factors.  In particular, we conclude that $\theta_Y\colon M\to k$ factors through $\sigma\colon M\to \Lambda(M)$.  Also the multiplicity of $S_X$ in $P_Y$ is the multiplicity of $S_X$ in the (semisimple) representation of $\Lambda(M)$ with character $\theta_Y$.

It is well known, cf.~\cite{Stanley}, that the primitive idempotent of $k\Lambda(M)$ corresponding to $S_X$ is \[\eta_X = \sum_{Z\leq X}Z\cdot \mu(Z,X).\]   It follows from standard representation theory, cf.~\cite{curtis}, that the multiplicity of $S_X$ in $\theta_Y$ is given by $\theta_Y(\eta_X)$, where $\theta_Y$ is viewed as a character of $\Lambda(M)$.  As $\sigma(e_Z)=Z$ for $Z\in \Lambda(M)$, it follows that
\[c_{X,Y} = \sum_{Z\leq X}\theta_Y(e_Z)\cdot \mu(Z,X).\]  But the identification of $P_Y$ with $k\til L_Y$ in Theorem~\ref{projindec} implies that $\theta_Y(e_Z)$ is the number of fixed points of $e_Z$ for its action on $\til L_Y$ by partial transformations, which is none other than $m_{Z,Y}$. This establishes the theorem.
\end{proof}

It is immediate that Theorem~\ref{Cartan} reduces to the result of~\cite{Saliola} for left regular bands.  Let us show how it reduces to the result of~\cite{Jtrivialpaper} for $\J$-trivial monoids. So suppose that $M$ is $\J$-trivial.
By M\"obius inversion, it follows that the $c_{X,Y}$ are characterized by the equation
\begin{equation}\label{deriveJtrivialCartan}
 m_{X,Y} = \sum_{Z\leq X}c_{Z,Y}.
\end{equation}
One has that $n\in \til L_Y$ if and only if $e_R(n)=e_Y$  and that $e_Xn=n$ if and only if $e_X\geq e_L(n)$.  Thus
\begin{equation*}
 m_{X,Y} = \sum_{Z\leq X}\left|\{n\in M\mid e_L(n)=e_Z,\ e_R(n)=e_Y \}\right|.
\end{equation*}
Therefore, we recover the result of~\cite{Jtrivialpaper} \[c_{X,Y}=|\{n\in M\mid e_L(n)=e_X,e_R(n)=e_Y \}|.\]  Notice that our proof is primitive idempotent-free.

\section{The quiver of a rectangular monoid}
We give here a procedure to compute the quiver of the algebra $kM$ of a rectangular monoid $M$ when $M$ splits over $k$ and the characteristic of $k$ does not divide the order of any maximal subgroup of $M$.  Even if the characteristic is bad, we still reduce the computation of the quiver of $kM$ to the representation theory of groups, only in this case we are dealing with the modular representation theory of groups.

Our results generalize those of~\cite{Saliola,rrbg,Jtrivialpaper}.  We retain the notation of Section~\ref{complexalgebra}.  We assume that $k$ is a splitting field for the maximal subgroups of $M$, and hence for $M$.

For each $X\in \Lambda(M)$, by the previous sections, we have a homomorphism $\rho_X\colon \CCC M\to \CCC G_X$ given on $M$ by
\[\rho_X(m)=\begin{cases} e_Xme_X & \sigma(m)\geq X\\ 0 & \text{else.}\end{cases}\] This is a slight abuse of notation, but should cause no confusion.  If $U$ is a simple $\CCC G_X$-module and $V$ is a simple $\CCC G_Y$-module, then $\Hom_k(U,V)$ is in fact a $\CCC [G_Y\times G_X^{op}]$-module and the $M\times M^{op}$-module structure is via the homomorphism $\CCC [M\times M^{op}]\to \CCC [G_Y\times G_X^{op}]$ induced by $\rho_Y\times \rho_X^{op}$.  Thus, to compute $\Ext^1_{kM}(U,V)\cong \Der(M,\Hom_k(U,V))/\IDer(M,\Hom_k(U,V))$, it suffices to compute $\Der(M,A)/\IDer(M,A)$ for any $G_Y\times G_X^{op}$-module $A$, viewed as an $M$-bimodule.  We fix such $A$ for the remainder of the section.
Our method of computing $H^1(M,A)$ is as follows.

The construction $A\mapsto \Der(M,A)/\IDer(M,A)$ is a functor from the category $\module{k[G_Y\times G_X^{op}]}$ to $k$-vector spaces.  We show that this functor is representable, and explicitly construct the representing object, whenever $X\neq Y$ or the characteristic of $k$ does not divide $|G_Y|\cdot|G_X|$.  In other words, we construct a $\CCC [G_Y\times G_X^{op}]$-module $V_{X,Y}$, depending only on $X,Y$, such that $\Der(M,A)/\IDer(M,A)\cong \Hom_{\CCC [G_Y\times G_X^{op}]}(V_{X,Y},A)$ (and the isomorphism is natural in $A$, although we do not prove it).  Of course, a $G_Y\times G_X^{op}$-module is the same thing as a $G_Y\times G_X$-module and hence in characteristic zero one can immediately compute $\dim \Hom_{\CCC [G_Y\times G_X^{op}]}(V_{X,Y},A)$ as an inner product of characters.

If $U$ is a simple $\CCC G_X$-module and $V$ is a simple $\CCC G_Y$-module, then the $k[G_Y\times G_X]$-module corresponding to $\Hom_k(U,V)$ is $V\otimes_k U^*$, which is a simple $k[G_Y\times G_X]$-module.  Moreover, every simple $k[G_Y\times G_X]$-module is of this form for some $U$ and $V$.  Thus computing the quiver of $kM$, when the characteristic of $k$ does not divide the order of any maximal subgroup of $M$, is equivalent to decomposing the $k[G_Y\times G_X]$-modules $V_{X,Y}$ into irreducibles for each $X$ and $Y$. In principal, this can be done from the character tables of $G_X$ and $G_Y$, at least in characteristic zero.

The construction of $V_{X,Y}$ breaks into four cases: $X\nleq Y$ and $Y\nleq X$; $X=Y$; $X<Y$; and $Y<X$. The last two cases are dual (i.e., one can replace $M$ by $M^{op}$ and perform the computation) and so we only handle the first three cases.

There are, however, some properties of derivations that are valid regardless of how $X$ and $Y$ relate.  We include some of them in the following proposition.  Any derivation $d\colon M\to A$ extends linearly to a mapping $d\colon kM\to A$ satisfying the Leibniz rule $d(ab)=ad(b)+d(a)b$ for all $a,b\in kM$.

\begin{Prop}\label{trivderprops}
Let $d\colon M\to A$ be a derivation. Then:
\begin{enumerate}
\item\label{trivderprops.1} $d(xyz) = xyd(z)+xd(y)z+d(x)yz$;
\item\label{trivderprops.2} $d(m) = d(e_Yme_X)-md(e_X) -d(e_Y)m$;
\item\label{trivderprops.3} $d(e_Ye_X) = d(e_X)+d(e_Y)$;
\item\label{trivderprops.4} $d(z)=0$ for $z\in \nup Y\nup X$;
\item\label{trivderprops.5} $d(e_Ymne_X)+d(e_Yme_Ye_Xne_X)-d(e_Yme_Yne_X)-d(e_Yme_Xne_X)=0$ for all $m,n\in M$.
\end{enumerate}
\end{Prop}
\begin{proof}
Trivially, $d(xyz)=xd(yz)+d(x)yz=xyd(z)+xd(y)z+d(x)yz$ yielding the first item.
The second item follows from the first because
\begin{align*}
d(e_Yme_X) &= e_Ymd(e_X)+e_Yd(m)e_X+d(e_Y)me_X\\  &= md(e_X)+d(m)+d(e_Y)m.
\end{align*}
The third item is $d(e_Ye_X) = e_Yd(e_X)+d(e_Y)e_X= d(e_X)+d(e_Y)$.  The fourth item follows because if $z=mn$ with $m\in \nup Y$ and $n\in \nup X$, then $d(z)=d(mn)=md(n)+d(m)n=0$.  The final item is equivalent to the equation  \[d\left(e_Ym(1-e_Y)(1-e_X)ne_X\right)=0,\] which follows because $(1-e_Y)A=0=A(1-e_X)$.
This completes the proof.
\end{proof}

It follows from the proposition that a derivation is for all practical purposes determined by its effect on $\AIrr_{\mathscr K(M)}(e_X,e_Y)=e_YMe_X\setminus \nup Y\nup X$.  This in turn is a $G_Y\times G_X^{op}$-set.  It is this feature that we shall exploit.

\subsection{The case $X,Y$ are incomparable}
Suppose now that $X,Y$ are incomparable and that $A$ is a $G_Y\times G_X^{op}$-bimodule.  Let $W_{X,Y}$ be the subspace of $k[e_YMe_X]$ spanned by $\nup Y \nup X\cap e_YMe_X$ and the elements of the form:
\begin{itemize}
\item [($\dagger$)] $e_Ymne_X+e_Yme_Ye_Xne_X-e_Yme_Yne_X-e_Yme_Xne_X$ with $m,n\in M$.
\end{itemize}
Note that the element in ($\dagger$) can be written in the more succinct form \[e_Ym(1-e_Y)(1-e_X)ne_X.\]

The proof of the next proposition is straightforward and so we omit it.

\begin{Prop}
One has $W_{X,Y}$ is a $G_Y\times G_X^{op}$-submodule of $k[e_YMe_X]$.
\end{Prop}
%\begin{proof}
%Trivially, $\nup Y\nup X\cap e_YMe_X$ is $G_Y\times G_X^{op}$-invariant.  If $m,n\in M$, $g\in G_Y$ and %$h\in G_X$, then putting $m'=gm$ and $n'=nh$, we have that
%\begin{align*}
%g(e_Ymne_X+e_Yme_Ye_Xne_X-e_Yme_Yne_X-e_Yme_Xne_X)h \\ = %e_Ym'n'e_X+e_Ym'e_Ye_Xn'e_X-e_Ym'e_Yn'e_X-e_Ym'e_Xn'e_X.
%\end{align*}
%It follows that $W_{X,Y}$ is $G_Y\times G_X^{op}$-invariant.
%\end{proof}

We put $V_{X,Y}=k[e_YMe_X]/W_{X,Y}$.  We now show that $V_{X,Y}$ represents outer derivations as a functor from $G_Y\times G_X^{op}$-modules to $k$-vector spaces.

\begin{Prop}\label{incomparablecase}
One has $H^1(M,A)\cong \Hom_{k[G_Y\times G_X^{op}]}(V_{X,Y},A)$.
\end{Prop}
\begin{proof}
Let $d\in \Der(M,A)$.  We can extend $d$ to $kM$ linearly.  Proposition~\ref{trivderprops} directly implies $d$ annihilates $W_{X,Y}$.  Thus there is a well-defined induced map $\ov d\colon V_{X,Y}\to A$.  Write $[z]$ for the class of $z\in e_YMe_X$ in $V_{X,Y}$.  Note that $e_YMe_X\subseteq \nup Y\cap \nup X$ because $X,Y$ are incomparable.

We claim that $\ov d$ is $G_Y\times G_X^{op}$-equivariant.  Let $g\in G_Y$ and $h\in G_X$.  Then $\ov d(g[z]h) = gzd(h)+gd(z)h+d(g)zh=gd(z)h=gd([z])h$ because $z\in \nup Y\cap \nup X$.

Define $\Phi\colon \Der(M,A)\to \Hom_{k[G_Y\times G_X^{op}]}(V_{X,Y},A)$ by $\Phi(d)=\ov d$.  First we show that $\ker \Phi=\IDer(M,A)$.  If $a\in A$ and $d_a$ is the corresponding inner derivation, then for $z\in e_YMe_X$ one has $d_a(z) = za-az=0$, as $z\in \nup Y\cap \nup X$.  Thus $\ov{d_a}=0$.   Next, suppose that $\ov d=0$.  Proposition~\ref{trivderprops}\eqref{trivderprops.3} then implies $0=\ov d([e_Ye_X]) = d(e_X)+d(e_Y)$.  Set $a=d(e_Y)=-d(e_X)$.  Then Proposition~\ref{trivderprops}\eqref{trivderprops.2} yields, for $m\in M$,
\[d(m) = \ov d([e_Yme_X])-md(e_X) -d(e_Y)m= ma-am\] and so $d$ is inner.

It remains to prove $\Phi$ is an epimorphism.  Let $f\in \Hom_{k[G_Y\times G_X^{op}]}(V_{X,Y},A)$.  Define $d\colon M\to A$ by $d(m) = f([e_Yme_X])-mf([e_Ye_X])$.  We must show that $d$ is a derivation.  Let $m,n\in M$.  Then \[d(mn) = f([e_Ymne_X])-mnf([e_Ye_X])\] whereas
\begin{align*}
md(n)+d(m)n =& mf([e_Yne_X])-mnf([e_Ye_X])\\  &+ f([e_Yme_X])n-mf([e_Ye_X])n.
\end{align*}
Thus it suffices to show that
\begin{equation}\label{derincomparablecase}
f([e_Ymne_X]) = mf([e_Yne_X])+ f([e_Yme_X])n-mf([e_Ye_X])n.
\end{equation}
We have four cases.
\begin{Case}
Suppose $\sigma(m)\ngeq Y$ and $\sigma(n)\ngeq X$.  Then $e_Ymne_X\in \nup Y\nup X$ and so both sides of \eqref{derincomparablecase} are $0$.
\end{Case}
\begin{Case}
Suppose $\sigma(m)\geq Y$ and $\sigma(n)\ngeq X$.  Then $[(e_Yme_Ye_X)ne_X]=0=[(e_Yme_X)ne_X]$ and so $[e_Ymne_X]=[e_Yme_Yne_X]$ by definition of $W_{X,Y}$.  Since $\sigma(n)\ngeq X$, the right hand side of \eqref{derincomparablecase} is then $mf([e_Yne_X]) = f([e_Yme_Yne_X])=f([e_Ymne_X])$ by $G_Y\times G_X^{op}$-equivariance of $f$.
\end{Case}
\begin{Case}
The case $\sigma(m)\ngeq Y$ and $\sigma(n)\geq X$ is dual to the previous one.
\end{Case}
\begin{Case}
Assume $\sigma(m)\geq Y$ and $\sigma(n)\geq X$.  By the $G_Y\times G_X^{op}$-equivariance of $f$, the right hand side of \eqref{derincomparablecase} in this case is \[f([e_Yme_Yne_X]+[e_Yme_Xne_X]-[e_Yme_Ye_Xne_X])=f([e_Ymne_X])\] where the last equality uses the definition of $W_{X,Y}$.
\end{Case}
It follows that $d$ is a derivation.  Next, suppose that $z\in e_YMe_X$.  Then $\ov d([z]) = d(z) = f([e_Yze_X])-zf([e_Ye_X]) = f([z])$ because $z\in \nup Y$ and $e_Yze_X=z$.  This completes the proof that $\Phi$ is surjective, thereby establishing the proposition.
\end{proof}

\subsection{The case $X=Y$}
This case differs from the other cases in that we do not in general obtain a representing object for $H^1(M,A)$, although we do so if the characteristic of the field does not divide the order of the maximal subgroup $G_X$.

\begin{Def}\label{defsimeeqcase}
Let $I_X=e_XMe_X\setminus G_X$.  Then $I_X$ is an ideal of $e_XMe_X$, and so in particular is $G_X\times G_X^{op}$-invariant.
Define $\sim$ to be the least equivalence relation on $I_X$ such that:
\begin{enumerate}
\item\label{defsimeeqcase.1} $e_Xmne_X\sim e_Xme_Xne_X$ for all $m,n\in M$ with $\sigma(mn)\ngeq X$;
\item\label{defsimeeqcase.2} $z\sim z'$ for all $z,z'\in (\nup X)^2$.
\end{enumerate}
\end{Def}

%Let us verify that $\sim$ is a $G_X\times G_X^{op}$-set congruence.
It is straightforward to verify that $\sim$ is a $G_X\times G_X^{op}$-set congruence.

\begin{Prop}\label{GXsetcong}
The equivalence relation $\sim$ is a $G_X\times G_X^{op}$-set congruence.
\end{Prop}
%\begin{proof}
%Define an equivalence relation $\frown$ on $I_X$ by $a\frown b$ if and only if $gah\sim gbh$ for all %$g,h\in G_X$. We claim that ${\frown}={\sim}$. Taking $g=e_X=h$ yields that ${\frown}\subseteq {\sim}$.  %Obviously, $\frown$ is a $G_X\times G_X^{op}$-set congruence.  To show ${\frown}={\sim}$, it remains to %show that $\frown$ satisfies (1)--(2) of Definition~\ref{defsimeeqcase}.  Property (2) trivially holds %because $(\nup X)^2$ is $G_X\times G_X^{op}$-invariant.
%
%Let $g,h\in G_X$.  If $m,n\in M$ with $\sigma(mn)\ngeq X$, then putting $m'=gm$ and $n'=nh$, we have that %$ge_Xmne_Xh=e_Xm'n'e_X$ and $ge_Xme_Xne_Xh=e_Xm'e_Xn'e_X$.  Since $\sigma(m'n')=\sigma(mn)\ngeq X$, it %follows that $ge_Xmne_Xh\sim ge_Xme_Xne_Xh$ for all $g,h\in G_X$.  Thus $e_Xmne_X\frown e_Xme_Xne_X$.  %This establishes (1).
%\end{proof}

Let $W_{X}$ be the submodule of $kI_X$ spanned by $(\nup X)^2\cap I_X$ and all differences $m-n$ with $m\sim n$. Set $V_{X,X} = kI_X/W_{X}$.  Because $(\nup X)^2$ is $G_X\times G_X^{op}$-invariant and $\sim$ is a $G_X\times G_X^{op}$-set congruence, $V_{X,X}$ is a $G_X\times G_X^{op}$-module.  We write $[m]$ for the image of $m\in I_X$ under the projection $kI_X\to V_{X,X}$.

The following lemma will be useful.

\begin{Lemma}\label{kille}
Let $d\colon M\to A$ be a derivation.  Then:
\begin{enumerate}
\item\label{kille.1} $d(e_X)=0$;
\item\label{kille.2} $d(e_Xme_X)=d(m)$ for all $m\in M$;
\item\label{kille.3} $d(me_Xn) = d(mn)$ for all $m,n\in M$.
\end{enumerate}
\end{Lemma}
\begin{proof}
From $d(e_X)=d(e_X^2) = e_Xd(e_X)+d(e_X)e_X=2d(e_X)$, we conclude $d(e_X)=0$. Therefore, $d(e_Xme_X) = e_Xmd(e_X)+e_Xd(m)e_X+d(e_X)me_X=d(m)$.  Finally, the last equality follows because $d(me_Xn)=me_Xd(n)+md(e_X)n+d(m)e_Xn=md(n)+d(m)n=d(mn)$.
\end{proof}

\begin{Prop}
Let $A$ be a $G_X\times G_X^{op}$-module.  Then
\[H^1(M,A)\cong H^1(G_X,A)\oplus \Hom_{k[G_X\times G_X^{op}]}(V_{X,X},A).\]  In particular, if the characteristic of $k$ does not divide the order of $G_X$, then the isomorphism \[H^1(M,A)\cong \Hom_{k[G_X\times G_X^{op}]}(V_{X,X},A)\] holds.
\end{Prop}
\begin{proof}
Let $d\colon M\to A$ be a derivation.  Obviously, $d|_{G_X}\in \Der(G_X,A)$.  Next we claim that, for $x,y\in I_X$, one has that $x\sim y$ implies $d(x)=d(y)$.  To this effect, define an equivalence relation $\equiv$ on $I_X$ by $x\equiv y$ if $d(x)=d(y)$.  We must show that $\equiv$ satisfies \eqref{defsimeeqcase.1} and \eqref{defsimeeqcase.2} of Definition~\ref{defsimeeqcase}.  Proposition~\ref{trivderprops}\eqref{trivderprops.4} implies that $d$ annihilates $(\nup X)^2$ and so \eqref{defsimeeqcase.2} is satisfied.  Lemma~\ref{kille}\eqref{kille.2} implies that \eqref{defsimeeqcase.1} is satisfied.

Since $d$ annihilates $(\nup X)^2$ by Proposition~\ref{trivderprops}, it follows that $d$ induces a $k$-linear mapping $\ov d\colon V_{X,X}\to A$.  We check that it is $G_X\times G_X^{op}$-equivariant.  Indeed, if $g,h\in G_X$, then for $m\in I_X$, one has $\sigma(m)\ngeq X$ and so \[\ov d(g[m]h) = gmd(h)+gd(m)h+d(g)mh=gd(m)h=g\ov d([m])h.\]

Thus we may define a $k$-linear mapping \[\Phi\colon \Der(M,A)\to H^1(G_X,A)\oplus \Hom_{k[G_X\times G_X^{op}]}(V_{X,X},A)\] by $\Phi(d) = (d|_{G_X}+\IDer(G_X,A),\ov d)$ where we identify $H^1(G_X,A)$ with $\Der(G_X,A)/\IDer(G_X,A)$.

First we verify that $\ker \Phi=\IDer(M,A)$.  If $a\in A$ and $d_a$ is the corresponding inner derivation, then, for $g\in G_X$, one has $d_a(g)=ga-ag$ and so $d_a|_{G_X}$ is inner.  Also, if $m\in I_X$, then $\ov {d_a}([m]) = d_a(m) = ma-am=0$ because $\sigma(m)\ngeq X$.  Thus $\IDer(M,A)\subseteq \ker \Phi$.  Suppose next that $d\in \ker \Phi$.  Then $d|_{G_X} = d_a$ for some $a\in A$.  We claim $d=d_a$ on $M$.  Indeed, if $\sigma(m)\geq X$, then $e_Xme_X\in G_X$ and so applying Lemma~\ref{kille} we conclude that $d(m)=d(e_Xme_{X}) = e_Xme_Xa-ae_Xme_X=ma-am$.  On the other hand, if $\sigma(m)\ngeq X$, then $e_Xme_X\in I_X$.  Thus $d(m)=d(e_Xme_X)=\ov d([e_Xme_X])=0$, where we have again used Lemma~\ref{kille}.  But $ma-am=0$ because $\sigma(m)\ngeq X$.  Thus $d=d_a$.

It remains to show that $\Phi$ is onto.  Let $D\colon G_X\to A$ be a derivation and let $f\colon V_{X,X}\to A$ be $G_X\times G_X^{op}$-equivariant.  Define $d\colon M\to A$ by
\[d(m) = \begin{cases} D(e_Xme_X) & \text{if}\ \sigma(m)\geq X\\ f([e_Xme_X]) & \text{else.}\end{cases}\]  Let us check that $d$ is a derivation.  Let $m,n\in M$.

\setcounter{Case}{0}
\begin{Case}
First suppose that $\sigma(m),\sigma(n)\ngeq X$. Then $d(mn) = f([e_Xmne_X])=0$ because $e_Xmne_X\in (\nup X)^2$.  Also, we have $md(n)+d(m)n=0$.
\end{Case}

\begin{Case}
Next suppose that $\sigma(m)\geq X$ and $\sigma(n)\ngeq X$.  We still have $d(mn)=f([e_Xmne_X])$.  Also \[md(n)+d(m)n=md(n) = e_Xme_Xf([e_Xne_X]) = f([e_Xme_Xne_X])\] because $e_Xme_X\in G_X$.  But the definition of $\sim$ implies that $[e_Xme_Xne_X]=[e_Xmne_X]$ and so $md(n)+d(m)n=d(mn)$.
\end{Case}

\begin{Case}
The case $\sigma(m)\ngeq X$ and $\sigma(n)\geq X$ is dual to the previous case.
\end{Case}

\begin{Case}
Next assume $\sigma(mn)\geq X$, then $e_Xme_X, e_Xne_X, e_Xmne_X\in G_X$.  Also $e_Xmne_X=e_Xme_Xne_X$ by Proposition~\ref{DOcharacterization}.  Thus we have
\begin{align*}
d(mn)&=D(e_Xmne_X)=D(e_Xme_Xne_X)\\ &= e_Xme_XD(e_Xne_X)+D(e_Xme_X)e_Xne_X\\ &=md(n)+d(m)n.
\end{align*}
\end{Case}

This completes the proof that $d$ is a derivation.  It remains to verify that $\Phi(d)=(D+\IDer(G_X,A),f)$.  If $g\in G_X$, then $d(g)=D(e_Xge_X)=D(g)$ and so $d|_{G_X}=D$.  If $m\in I_X$, then $\ov d([m]) = d(m) = f([e_Xme_X])=f([m])$.  Thus $\Phi(d) = (D+\IDer(G_X,A),f)$, as required.

The final statement of the proposition is a direct consequence of Proposition~\ref{grouphochiszero}.
\end{proof}

\subsection{The case $X$ and $Y$ are strictly comparable}
We only handle the case $X<Y$, since the case $Y<X$ can be handled by considering the dual monoid $M^{op}$.
So suppose now that $X<Y$ and $A$ is a $G_Y\times G_X^{op}$-module.   Without loss of generality, we may assume that $e_X < e_Y$.  Indeed, if we put $f=e_Ye_Xe_Y$, then $f^2=e_Ye_Xe_Ye_Xe_Y=e_Ye_Xe_Y$ by Proposition~\ref{DOcharacterization}. Also $\sigma(f)=X$ and $f<e_Y$.  Thus we can replace $e_X$ by $f$ to obtain the desired property.

\begin{Def}\label{definesim}
Let us define $\sim$ to be the least equivalence relation on $e_YMe_X$ such that:
\begin{enumerate}
\item\label{definesim.1} $e_Ymne_X\sim e_Yme_Yne_X$ for $m,n\in M$;
\item\label{definesim.2} if $\sigma(m)\ngeq Y$, then $e_Ymne_X\sim e_Yme_Xne_X$ for all $n\in M$;
\item\label{definesim.3} $z\sim z'$ for all $z,z'\in \nup Y\nup X$.
\end{enumerate}
\end{Def}

%The following proposition has a nearly identical proof to that of Proposition~\ref{GXsetcong} and so we %omit it.

It is routine to verify that $\sim$ is a $G_Y\times G_X^{op}$-set congruence.

\begin{Prop}
The equivalence relation $\sim$ is a $G_Y\times G_X^{op}$-set congruence.
\end{Prop}
%\begin{proof}
%Define an equivalence relation $\frown$ on $e_YMe_X$ by $a\frown b$ if and only if $gah\sim gbh$ for all $g\in G_Y$ and $h\in G_X$.  Taking $g=e_Y$ and $h=e_X$ shows that ${\frown}\subseteq {\sim}$.  Trivially, $\frown$ is a $G_Y\times G_X^{op}$-set congruence.  To show ${\frown}={\sim}$, it remains to show that $\frown$ satisfies (1)--(3) of Definition~\ref{definesim}.
%
%Let $(g,h)\in G_Y\times G_X$.  If $m,n\in M$, then putting $m'=gm$ and $n'=nh$, we have that $ge_Ymne_Xh=e_Ym'n'e_X$ and $ge_Yme_Yne_Xh=e_Ym'e_Yn'e_X$ and so $ge_Ymne_Xh\sim ge_Yme_Yne_Xh$ for all $(g,h)\in G_Y\times G_X$.  Thus $e_Ymne_X\frown e_Yme_Yne_X$.  This establishes (1).  For (2), let $\sigma(m)\ngeq Y$ and $n\in M$.  Let $m'=gm$ and $n'=nh$.  Then $\sigma(m')\ngeq Y$ and so $e_Ygmnhe_X= e_Ym'n'e_X\sim e_Ym'e_Xn'e_X=e_Ygme_Xnhe_X$. Thus $e_Ymne_X\frown e_Yme_Xne_X$ and so (2) holds.  That (3) holds is trivial since $\nup Y\nup X$ is $G_Y\times G_X^{op}$-invariant.
%\end{proof}

Notice that $G_X$ has the structure of a $G_Y\times G_X^{op}$-set, the left action of $G_Y$ coming via the restriction of the homomorphism $\rho_X\colon M_X\to G_X$ to $G_Y$.  In fact, $G_X\cup \{0\}\subseteq kG_X$ is a $G_Y\times G_X^{op}$-set as $0$ is fixed by all elements of $G_Y\times G_X^{op}$.  Clearly, the restricted mapping $\rho_X\colon e_YMe_X\to G_X\cup \{0\}\subseteq kG_X$ is $G_Y\times G_X^{op}$-equivariant.

\begin{Prop}\label{containedinkernelofrho}
Suppose that $a,b\in e_YMe_X$ satisfy $a\sim b$.  Then $\rho_X(a)=\rho_X(b)$.
\end{Prop}
\begin{proof}
Define an equivalence relation $\equiv$ on $e_YMe_X$ by $x\equiv y$ if $\rho_X(x)=\rho_X(y)$.  It suffices to show that $\equiv$ satisfies \eqref{definesim.1}--\eqref{definesim.3} of Definition~\ref{definesim}.  Since $\rho_X(\nup Y\nup X)=0$, \eqref{definesim.3} is trivially satisfied.  If $m,n\in M$, then $\rho_X(e_Ymne_X) = \rho_X(mn)=\rho_X(e_Yme_Yne_X)$ because $\rho_X(e_Y)=e_X=\rho_X(e_X)$.  Similarly, for \eqref{definesim.2}, one has $\rho_X(e_Ymne_X)=\rho_X(mn)=\rho_X(e_Yme_Xne_X)$.  This concludes the proof.
\end{proof}

Let $Z_{X,Y}$ be the submodule of $k(e_YMe_X)$ spanned by $\nup Y\nup X\cap e_YMe_X$ and all differences $m-n$ with $m\sim n$. Define $W_{X,Y} = k[e_YMe_X]/Z_{X,Y}$.  As $\nup Y\nup X$ is $G_Y\times G_X^{op}$-invariant and $\sim$ is a $G_Y\times G_X^{op}$-set congruence, it follows that $W_{X,Y}$ is a $G_Y\times G_X^{op}$-module.  We write $[m]$ for the image of $m\in e_YMe_X$ under the canonical homomorphism $e_YMe_X\to W_{X,Y}$.  Since $\rho_X(\nup Y\nup X)=0$, it follows from Proposition~\ref{containedinkernelofrho}, that there is a well-defined $kG_Y\times G_X^{op}$-module homomorphism $\ov{\rho_X}\colon W_{X,Y}\to kG_X$ induced by $\rho_X$.  Moreover, $\ov{\rho_X}$ is surjective and has a $G_X^{op}$-module splitting $\gamma\colon kG_X\to W_{X,Y}$ given by $h\mapsto [h]$ for $h\in G_X$. (This map is not $G_Y$-equivariant.)

Let $V_{X,Y} = \ker \ov{\rho_X}$.  Because $\gamma$ is a $k$-splitting, it follows that $V_{X,Y}$ has a $k$-vector space basis consisting of all elements of the form $[a]-\gamma\rho_X([a])$ with $a\in e_YMe_X\setminus (\nup Y\nup X\cup G_X)$ and $[a]\neq 0$.  We observe that $\gamma\rho_X(a) = [e_Xa]$.  Indeed, this is trivial if $\sigma(a)\geq X$ because in this case $\rho_X(a) = e_Xae_X=e_Xa$.  If $a\in e_YMe_X\cap \nup X$, then $\rho_X([a]) =0$.  On the other hand, $e_Xa\in \nup Y\nup X$ and so $[e_Xa]=0=\gamma(\rho_X([a]))$.  Thus $V_{X,Y}$ has a $k$-vector space basis consisting of the elements of the form $[a]-[e_Xa]$ with $a\in e_YMe_X\setminus (\nup Y\nup X\cup G_X)$ and $[a]\neq 0$.

We are now ready to state the main result of this subsection, namely that $V_{X,Y}$ represents outer derivations for $G_Y\times G_X^{op}$-modules $A$.

\begin{Prop}\label{comparablecase}
One has $H^1(M,A)\cong \Hom_{k[G_Y\times G_X^{op}]}(V_{X,Y},A)$.
\end{Prop}
\begin{proof}
Suppose that $d\colon M\to A$ is a derivation.  First we claim that $x\sim y$ implies $d(x)=d(y)$.  Define an equivalence relation $\equiv$ on $e_YMe_X$ by $x\equiv y$ if $d(x)=d(y)$.  We must show that $\equiv$ satisfies \eqref{definesim.1}--\eqref{definesim.3} of Definition~\ref{definesim}.  Property \eqref{definesim.3} is clear since $d$ annihilates $\nup Y\nup X$ by Proposition~\ref{trivderprops}\eqref{trivderprops.4}.  To prove Property \eqref{definesim.1}, first observe that $e_Yme_Ye_Xne_X=e_Yme_Xne_X$ because we have chosen $e_X < e_Y$.  Thus Proposition~\ref{trivderprops}\eqref{trivderprops.5} implies $d(e_Ymne_X)=d(e_Yme_Yne_X)$.  If $\sigma(m)\ngeq Y$, then we have, for $n\in M$, that $d(e_Ymne_X)=md(ne_X)+d(e_Ym)n=d(e_Ym)n$ and $d(e_Yme_Xne_X)=md(e_Xne_X)+d(e_Ym)n=d(e_Ym)n$.  Thus \eqref{definesim.2} is satisfied, as well.

As we've seen in the previous paragraph, $d$ annihilates $\nup Y\nup X$, and so $d$ descends to a well-defined mapping $W_{X,Y}\to A$, which can then be restricted to a mapping $\ov d\colon V_{X,Y}\to A$. We claim that $\ov d$ is $G_Y\times G_X^{op}$-equivariant.  Indeed, let $g\in G_Y$, $h\in G_X$ and $a\in e_YMe_X\setminus (\nup Y\nup X\cup G_X)$.  Then
\begin{align*}
\ov d(g([a]-[e_Xa])) &= d(ga)-d(ge_Xa)=gd(a)+d(g)a-gd(e_Xa)-d(g)a\\
                     &= g(d(a)-d(e_Xa))=g\ov d([a]-[e_Xa])\\
\ov d(([a]-[e_Xa])h) &= d(ah)-d(e_Xah) = ad(h)+d(a)h-e_Xad(h)-d(e_Xa)h\\
&= d(a)h-d(e_Xa)h= \ov d([a]-[e_Xa])h
\end{align*}
where the penultimate equality uses that $\sigma(a)\ngeq Y$ and $X<Y$.

Thus we have a $k$-linear mapping
\[\Phi\colon \Der(M,A)\to \Hom_{k[G_Y\times G_X^{op}]}(V_{X,Y},A)\] given by $\Phi(d)=\ov d$.  We claim that $\ker \Phi=\IDer(M,A)$.  Indeed, let $a\in A$ and let $d_a$ be the corresponding inner derivation.  Then, for $b\in e_YMe_X$, one computes
\[\ov d_a([b]-[e_Xb])=ba-ab-e_Xba+ae_Xb=-ab+ab=0\]
because $\sigma(b)\ngeq Y$ (as $X<Y$).  It follows that $\Phi(d_a)=0$.

Conversely, suppose that $d\in \ker \Phi$.  Trivially, $d(e_Y)=d(e_{Y}e_{Y}) = e_Yd(e_Y)+d(e_Y)e_Y=d(e_Y)+d(e_Y)=2d(e_Y)$ and so $d(e_Y)=0$.  From $\ov d=0$, we obtain, for $m\in M$,
\[0=\ov d([e_Yme_X]-[e_Xme_X]) = d(e_Yme_X)-d(e_Xme_X) = d(e_Yme_X)-d(e_X)m\] and so Proposition~\ref{trivderprops}\eqref{trivderprops.2} implies that $d(m)=d(e_X)m-md(e_X)-d(e_Y)m=d(e_X)m-md(e_X)$.  Thus $d$ is the inner derivation associated to $-d(e_X)$.  This completes the proof that $\ker \Phi=\IDer(M,A)$.  It remains to establish that $\Phi$ is an epimorphism.

Let $f\colon V_{X,Y}\to A$ be a $G_Y\times G_X^{op}$-module homomorphism.  Define a mapping $d\colon M\to A$ by
\[d(m) = f([e_Yme_X]-[e_Xme_X]).\]
Let us prove that $d$ is a derivation.  Let $m,n\in M$.  Then
\begin{equation}\label{prodmn}
d(mn) = f([e_Ymne_X]-[e_Xmne_X]).
\end{equation}
We have several cases.

\setcounter{Case}{0}
\begin{Case}
Suppose that $\sigma(m)\ngeq Y$ and $\sigma(n)\ngeq X$.  Then $e_Ymne_X, e_Xmne_X\in \nup Y\nup X$ and so $d(mn)=0$ by \eqref{prodmn}.  Also, $md(n)+d(m)n=0$.
\end{Case}

\begin{Case}
Assume that $\sigma(m)\geq Y$ and $\sigma(n)\ngeq X$.  Then $e_Xmne_X\in \nup Y\nup X$ and so \eqref{prodmn} becomes $d(mn)=f([e_Ymne_X])$.  On the other hand, since $md(n)+d(m)n=md(n)$ and $e_Xne_X\in \nup Y\nup X$ we have
\begin{align*}
md(n) &= mf([e_Yne_X]-[e_Xne_X])=mf([e_Yne_X])=e_Yme_Yf([e_Yne_X])\\
&= f([e_Yme_Yne_X]) = f([e_Ymne_X])=d(mn)
\end{align*}
since $[e_Yme_Yne_X]=[e_Ymne_X]$ by definition of $\sim$.  This completes this case.
\end{Case}

The next three cases assume that $\sigma(n)\geq X$.  In this context, the following formula is valid:
\begin{equation}\label{nabovecase}
\begin{split}
d(m)n &= f([e_Yme_X]-[e_Xme_X])n = f([e_Yme_X]-[e_Xme_X])e_Xne_X\\
&= f([e_Yme_Xne_X]-[e_Xme_Xne_X]).
\end{split}
\end{equation}

\begin{Case}
Suppose that $\sigma(m)\geq Y$ and $\sigma(n)\geq X$.  Then
\begin{equation}\label{mabove}
\begin{split}
md(n)&= mf([e_Yne_X]-[e_Xne_X])=e_Yme_Yf([e_Yne_X]-[e_Xne_X])\\
&= f([e_Yme_Yne_X]-[e_Yme_Xne_X])
\end{split}
\end{equation}
Since $\sigma(m)\geq Y> X$, Proposition~\ref{DOcharacterization} implies that $e_Xme_Xne_X=e_Xmne_X$.  The equality $[e_Yme_Yne_X]=[e_Ymne_X]$ holds by definition of $\sim$.  Taking this into account and adding equations \eqref{nabovecase} and \eqref{mabove} yields
\[md(n)+d(m)n=f([e_Ymne_X]-[e_Xmne_X])=d(mn),\]
as required.
\end{Case}

\begin{Case}
Assume that $\sigma(m)\geq X$, $\sigma(m)\ngeq Y$ and $\sigma(n)\geq X$.  Then $md(n)=0$ and so it suffices to show that the right hand sides of equations \eqref{prodmn} and \eqref{nabovecase} agree.  Proposition~\ref{DOcharacterization} implies that $[e_Xme_Xne_X]=[e_Xmne_X]$, whereas the definition of $\sim$ implies that $[e_Yme_Xne_X]=[e_Ymne_X]$.  Thus the right hand sides of \eqref{prodmn} and \eqref{nabovecase} agree, completing this case.
\end{Case}

\begin{Case}
Assume now that $\sigma(m)\ngeq X$ and $\sigma(n)\geq X$.  Again, we have $md(n)=0$ and so it suffices to show that the right hand sides of equations \eqref{prodmn} and \eqref{nabovecase} are equal.  In this case, $e_X(mne_X), e_X(me_Xne_X)\in \nup Y\nup X$ and so $d(mn)=f([e_Ymne_X])$ and $d(m)n = f([e_Yme_Xne_X])$.  But $[e_Yme_Xne_X]=[e_Ymne_X]$ by definition of $\sim$.  This completes this case.
\end{Case}

We have thus established that $d$ is a derivation.  Let us show that $\Phi(d)=f$.  Suppose that $a\in e_YMe_X\setminus (\nup Y\nup X\cup G_X)$.  Then $e_Yae_X=a=ae_X$ and so $\ov d([a]-[e_Xa])=d(a)-d(e_Xa)=f([a]-[e_Xa])-f([e_Ye_Xa]-[e_Xa])=f([a]-[e_Xa])$ because $e_Ye_X=e_X$.  This concludes the proof that $\Phi$ is an epimorphism, and hence the theorem is proved.
\end{proof}

We now state the main result of this paper.  It puts together what we have already proven.  See the discussion before Proposition~\ref{trivderprops} on how we go from $G_Y\times G_X^{op}$-modules to $G_Y\times G_X$-modules.

\begin{Thm}\label{thebigtheorem}
Let $M$ be a rectangular monoid and $k$ a splitting field for $M$ whose characteristic does not divide the order of any maximal subgroup of $M$.  Then $Q(kM)$ has vertex set $\coprod_{X\in \Lambda(M)}\Irr(G_X)$.    Let $X,Y\in \Lambda(M)$ and let $U\in \Irr(G_X), V\in \Irr(G_Y)$.  Let $V_{X,Y}$ be the $G_Y\times G_X$-module constructed above.  Then the number of arrows from $U$ to $V$ in the quiver of $kM$ is the multiplicity of $V\otimes_k U^*$ as an irreducible constituent of $V_{X,Y}$.
\end{Thm}

Theorem~\ref{thebigtheorem} simplifies when each maximal subgroup of $M$ is trivial.  The subclass of rectangular monoids with trivial maximal subgroups is known as $\pv{DA}$ in the semigroup theory literature~\cite{qtheor,Almeida:book}; it was first studied by Sch\"utzenberger in the context of automata theory~\cite{Schutznonambig}.

\begin{Cor}\label{aperiodiccase}
Let $M$ be a rectangular monoid with trivial maximal subgroups and $k$ a field.  Then $Q(kM)$ has vertex set $\Lambda(M)$.  The number of arrow from $X$ to $Y$ is the dimension of the vector space $V_{X,Y}$ constructed above.
\end{Cor}

\subsection{Proofs of the results of Section~\ref{stateresults}}
The remainder of this section is devoted to deriving the results of Section~\ref{stateresults} from Theorem~\ref{thebigtheorem}.  We begin with the proof of Theorem~\ref{orthogroupquiver}.

\begin{proof}[Proof of Theorem~\ref{orthogroupquiver}]
Let $U\in \Irr(G_X)$ and $V\in \Irr(G_Y)$.  Suppose first that $X$ and $Y$ are incomparable.  Then by regularity of $M$ we have $e_YMe_X\subseteq e_YMe_XMe_YMe_X\subseteq  \nup Y\nup X$.  Thus $V_{X,Y}=0$ and so we have no arrows from $U$ to $V$ in the incomparable case.  Next suppose that $X=Y$.  Then regularity implies that $I_X\subseteq I_XI_X\subseteq (\nup X)^2$.  Thus $V_{X,X}=0$, and so again there are no arrows from $U$ to $V$.

We now handle the case $X<Y$ as the case $Y<X$ is dual.  Assume $e_X<e_Y$. First observe that $e_YMe_X\setminus e_YL_{e_X}\subseteq \nup Y\nup X$ by regularity.  Thus it remains to verify that the equivalence relation $\equiv$ in Theorem~\ref{orthogroupquiver} on $e_YL_{e_X}$ agrees with $\sim$ from Definition~\ref{definesim}.  The equivalence of (a) with \eqref{definesim.1} of Definition~\ref{definesim} is straightforward.  Clearly \eqref{definesim.2} of Definition~\ref{definesim} implies (b) because $x=e_Yexe_X\sim e_Yee_Xxe_X=e_Yee_Xx'e_X\sim e_Yex'e_X=x'$ if $e\in E(e_YMe_Y)$ with $\sigma(e)\neq e_Y$.

To see that (b) implies \eqref{definesim.2} we proceed as follows.  Suppose $\sigma(m)\ngeq Y$.  If $\sigma(m)\ngeq X$, then $e_Ymne_X$ and $e_Yme_Xne_X$ belong to $\nup Y\nup X$ by regularity, and so we may assume $\sigma(m)\geq X$.  If $\sigma(m)=X$, then $MmM=Me_XM$ by regularity and so $e_Ymne_X=e_Yme_Xne_X$ by Proposition~\ref{DOcharacterization}.  Thus we may assume $\sigma(m)>X$ but $\sigma(m)\ngeq Y$.  Let $e=(e_Yme_Y)^{\omega}$.  Then $e\in E(e_YMe_Y)$.  Observe that $ee_Ym=(e_Yme_Y)^{\omega}e_Ym=(e_Ym)^{\omega}e_Ym=e_Ym$ as $M$ is an orthogroup and hence satisfies $x^{\omega}x=x$ for all $x\in M$. It follows that $ee_Ymne_X=e_Ymne_X$ and $ee_Yme_Xne_X=e_Yme_Xne_X$.
Also observe that $e_Xe_Ymne_X=e_Xmne_X=e_Xe_Yme_Xne_X$ by Proposition~\ref{DOcharacterization}.  It follows that $e_Ymne_X\equiv e_Yme_Xne_X$.  Thus \eqref{definesim.2} of Definition~\ref{definesim} is equivalent to (b) in Theorem~\ref{orthogroupquiver}.  This completes the proof.
\end{proof}

Next we turn to the case of monoids in $\pv {DG}$.

\begin{proof}[Proof of Theorem~\ref{quiverDG}]
Again, we assume $U\in \Irr(G_X)$ and $V\in \Irr(G_Y)$.  We proceed case by case.  Suppose first that $X,Y$ are incomparable.  We claim that, for all $m,n\in M$,
\begin{equation}\label{annoyingrelation}
e_Ymne_X+e_Yme_Ye_Xne_X-e_Yme_Yne_X-e_Yme_Xne_X
\end{equation}
is in the span of $\nup Y\nup X\cap e_YMe_X$.  Indeed, if $\sigma(m)\geq Y$, then since $M\in \pv{DG}$ we have $e_Ym=e_Yme_Y$ and so \eqref{annoyingrelation} is $0$. The case $\sigma(n)\geq X$ is dual.  If $m\in \nup Y$ and $n\in \nup X$, then trivially \eqref{annoyingrelation} is in the span of $\nup Y\nup X$.  As any element of $e_YMe_X\setminus \nup Y\nup X$ must be a null element when $X$ and $Y$ are incomparable, we see that $V_{X,Y}\cong k[\Null(e_YMe_X)\setminus \nup Y\nup X]$.  But this is $k[\Irrm_{\mathscr K(M)}(e_X,e_Y)]$ by Proposition~\ref{simplifyirred}.

Next assume that $X=Y$.  If $m,n\in \nup X$, then trivially $e_Xmne_X$, $e_Xme_Xne_X$ belong to $(\nup X)^2\cap I_X$.  If $\sigma(m)\geq X$, then $e_Xm=e_Xme_X$ and so $e_Xmne_X=e_Xme_Xne_X$.  Similarly, if $\sigma(n)\geq X$, then $e_Xmne_X=e_Xme_Xne_X$.  Thus \eqref{defsimeeqcase.1} of Definition~\ref{defsimeeqcase} can be removed.  Clearly, $I_X\setminus (\nup X)^2=\Null(e_XMe_X)\setminus (\nup X)^2$.  Thus $V_{X,Y} = k[\Null(e_YMe_X)\setminus \nup Y\nup X]=k[\Irrm_{\mathscr K(M)}(e_X,e_Y)]$ in this case as well.

The cases $X<Y$ and $Y<X$ are dual, so we just handle $X<Y$.  One necessarily has in this case that $e_X<e_Y$.  First observe that if $\sigma(m)\geq Y$, then $e_Ym=e_Yme_Y$ and so \eqref{definesim.1} of Definition~\ref{definesim} holds.  If $m\in \nup Y$, $n\in \nup X$, then $e_Ymne_X$, $e_Yme_Yne_X$, $e_Yme_Xne_X$ all belong to $\nup Y\nup X\cap e_YMe_X$.  If $m\in \nup Y$ and $\sigma(n)\geq X$, then $e_Xne_X=ne_X$ and so $e_Ymne_X=e_Yme_Xne_X=e_Yme_Ye_Xne_X=e_Yme_Yne_X$.  Thus \eqref{definesim.1} and \eqref{definesim.2} of Definition~\ref{definesim} are automatic. In particular, no element of $e_YMe_X\setminus \nup Y\nup X$ is identified by $\sim$ with an element of $\nup Y\nup X$. Thus, since $M\in \pv {DG}$, it follows that $V_{X,Y}$ has $k$-basis consisting of all elements of the form $a$ with $a\in e_YMe_X\setminus (\nup Y\nup X\cup G_X)=\Null(e_YMe_X)\setminus \nup Y\nup X$.  Thus, once again, $V_{X,Y}\cong k[\Null(e_YMe_X)\setminus \nup Y\nup X]=k[\Irrm_{\mathscr K(M)}(e_X,e_Y)]$.  This completes the proof.
\end{proof}

It remains to prove Theorem~\ref{Rtrivialquiver}.

\begin{proof}[Proof of Theorem~\ref{Rtrivialquiver}]
Let $X,Y\in \Lambda(M)$.  Once more we proceed case by case.   Suppose first that $X,Y$ are incomparable.  We claim that modulo $k[\nup Y\nup X\cap e_YMe_X]$, the element in \eqref{annoyingrelation} above is either $0$ or of the form $e_Ymne_X-e_Yme_X$.  Indeed, if $m\in\nup Y$, $n\in\nup X$, then trivially \eqref{annoyingrelation} belongs to $k[\nup Y\nup X\cap e_YMe_X]$.  Suppose next that $\sigma(m)\geq Y$.  Then $e_Ym=e_Yme_Y$ and so \eqref{annoyingrelation} is $0$.  If $m\in \nup Y$ and $\sigma(n)\geq X$, then $e_Xne_X=e_X$ and $e_Yme_Ye_Xne_X$, $e_Yme_Yne_X$ belong to $\nup Y\nup X$.  Thus \eqref{annoyingrelation} reduces to $e_Ymne_X-e_Yme_X$ modulo $k[\nup Y\nup X\cap e_YMe_X]$.  We conclude that $V_{X,Y}\cong kZ_{X,Y}$ and so $\dim V_{X,Y}=|Z_{X,Y}|$.

Next, suppose that $X=Y$.  First we show that $\sim$ and $\equiv$ coincide.  If $m,n\in \nup X$, then $e_Xmne_X,e_Xme_Xne_X\in (\nup X)^2$ and so are equivalent under both equivalence relations.  If $\sigma(m)\geq X$, then $e_Xm=e_Xme_X$ and so $e_Xmne_X=e_Xme_Xne_X$.  If $m\in \nup X$ and $\sigma(n)\geq X$, then $e_Xne_X=e_X$ and so $e_Xme_Xne_X=e_Xme_X$.  Thus \eqref{defsimeeqcase.1} of Definition~\ref{defsimeeqcase} reduces in this case to the definition of $\equiv$. It follows that $V_{X,X}\cong k[Z_{X,X}\setminus\{[e_X]\}]$ where $[e_X]$ is the class of $e_X$ (which is a singleton).  Thus $\dim V_{X,X}=|Z_{X,X}|-1$.

The next case is $X<Y$.  We may assume that $e_X<e_Y$.  It suffices to show that $\equiv$ coincides with $\sim$.  If $m\in \nup Y$ and $n\in \nup X$, then $e_Ymne_X$, $e_Yme_Yne_X$ and $e_Yme_Xne_X$ all belong to $\nup Y\nup X$.  If $\sigma(m)\geq Y$, then $e_Ym=e_Yme_Y$ and so $e_Ymne_X = e_Yme_Yne_X$, whence \eqref{definesim.1} of Definition~\ref{definesim} holds.  If $\sigma(m)\ngeq Y$ and $\sigma(n)\geq X$, then also $\sigma(e_Yn)\geq X$ and $e_Xe_Yne_X=e_X=e_Xne_X=e_X$.  Thus $\sim$ reduces to $\equiv$.  It follows that $W_{X,Y}\cong kZ_{X,Y}$.  Going to $V_{X,Y}$ cuts the dimension down by $1$, yielding the result.

Finally, we deal with the case $Y<X$.  Notice that this case is not dual to $X<Y$ because $M^{op}$ need not be $\R$-trivial.  Nonetheless, the computation is similar to the previous one, but simpler.  We may assume $e_Y<e_X$.  The dual of $\sim$ from Definition~\ref{definesim} is given by
\begin{enumerate}
\item\label{noname.1} $e_Ymne_X\sim e_Yme_Xne_X$ for $m,n\in M$;
\item\label{noname.2} if $\sigma(n)\ngeq X$, then $e_Ymne_X\sim e_Yme_Yne_X$ for all $m\in M$;
\item\label{noname.3} $z\sim z'$ for all $z,z'\in \nup Y\nup X$.
\end{enumerate}

We aim to show for an $\R$-trivial monoid, that the restriction of $\sim$ to $Z_{X,Y}=e_YMe_X\setminus \nup Y\nup X$ coincides with $\equiv$.
If $\sigma(m)\geq Y$, then $e_Ym=e_Yme_Y$ and so Property \eqref{noname.2} is satisfied.  If $\sigma(m)\ngeq Y$ and $\sigma(n)\ngeq X$, then $e_Ymne_X, e_Yme_Yne_X\in \nup Y\nup X$.  Thus we can remove \eqref{noname.2} from the definition of $\sim$.  If $\sigma(m)\geq Y$, then $e_Ym=e_Yme_Y$ and so $e_Ymne_X=e_Yme_Yne_X=e_Yme_Ye_Xne_X=e_Yme_Xne_X$ because $e_Y<e_X$.  Hence \eqref{noname.1} is only of interest when $\sigma(m)\ngeq Y$ and $\sigma(n)\geq X$.  In this case $e_Xne_X=e_X$ and so \eqref{noname.1} boils down to the definition of $\equiv$.  The remainder of the proof is handled as in the previous case.
\end{proof}

%\bibliographystyle{abbrv}
%\bibliography{standard2}

\end{document}